\documentclass{amsart}
\usepackage{amsmath,amssymb,amsthm}
\usepackage{bm}
\usepackage{mathrsfs}
\usepackage{stmaryrd}
\usepackage[colorlinks=true]{hyperref}
\usepackage[all]{xy}

\newcommand\simto{\sto{\sim}}
\newcommand\lra{\longrightarrow}
\newcommand\ov{\overline}
\newcommand\half{\frac{1}{2}}
\font\cyr=wncyr10
\newcommand\Sha{\hbox{\cyr X}}
\newcommand\sto[1]{\stackrel{#1}{\longrightarrow}}
\newcommand\set[1]{{\left\{#1\right\}}}
\newcommand\pair[1]{\langle{#1}\rangle}
\newcommand\svec[2]{\begin{pmatrix}#1\\#2\end{pmatrix}}
\newcommand\svecc[3]{\begin{pmatrix}#1\\#2\\#3\end{pmatrix}}
\newcommand\leg[2]{\Bigl(\frac{{#1}}{#2}\Bigr)}
\newcommand\aleg[2]{\Bigl[\frac{{#1}}{#2}\Bigr]}
\newcommand\rmT{{\mathrm{T}}}
\newcommand\bfa{{\mathbf{a}}}
\newcommand\bfb{{\mathbf{b}}}
\newcommand\bfd{{\mathbf{d}}}
\newcommand\bfu{{\mathbf{u}}}
\newcommand\bfx{{\mathbf{x}}}
\newcommand\bfy{{\mathbf{y}}}
\newcommand\bfz{{\mathbf{z}}}
\newcommand\bfA{{\mathbf{A}}}
\newcommand\bfB{{\mathbf{B}}}
\newcommand\bfD{{\mathbf{D}}}
\newcommand\bfF{{\mathbf{F}}}
\newcommand\bfG{{\mathbf{G}}}
\newcommand\bfI{{\mathbf{I}}}
\newcommand\bfM{{\mathbf{M}}}
\newcommand\bfN{{\mathbf{N}}}
\newcommand\bfO{{\mathbf{O}}}
\newcommand\bfR{{\mathbf{R}}}
\newcommand\BA{{\mathbb{A}}}
\newcommand\BF{{\mathbb{F}}}
\newcommand\BQ{{\mathbb{Q}}}
\newcommand\BR{{\mathbb{R}}}
\newcommand\BZ{{\mathbb{Z}}}
\newcommand\CA{{\mathcal{A}}}
\providecommand\CD{{\mathcal{D}}}
\newcommand\CE{{\mathcal{E}}}
\newcommand\CH{{\mathcal{H}}}
\newcommand\CM{{\mathcal{M}}}
\newcommand\CO{{\mathcal{O}}}
\newcommand\CP{{\mathcal{P}}}
\newcommand\msa{\mathscr{A}}
\newcommand\msb{\mathscr{B}}
\newcommand\msi{\mathscr{I}}
\newcommand\msp{\mathscr{P}}
\newcommand\msq{\mathscr{Q}}
\newcommand\mst{\mathscr{T}}
\newcommand\fa{{\mathfrak{a}}}
\newcommand\fc{{\mathfrak{c}}}
\newcommand\ff{{\mathfrak{f}}}
\newcommand\fp{{\mathfrak{p}}}
\newcommand\diag{{\mathrm{diag}}}
\renewcommand\Im{{\mathrm{Im}\,}}
\newcommand\Ker{{\mathrm{Ker}\,}}
\newcommand\Li{{\mathrm{Li}}}
\renewcommand\mod{\, \mathrm{mod}\, }
\newcommand\rank{{\mathrm{rank}\,}}
\renewcommand\Re{{\mathrm{Re}\,}}
\newcommand\Sel{{\mathrm{Sel}}}
\newcommand\tor{{\mathrm{tor}}}

\renewcommand{\le}{\leqslant}
\renewcommand{\ge}{\geqslant}

\newtheorem{theorem}{Theorem}[section]

\newtheorem{lemma}[theorem]{Lemma}
\newtheorem{proposition}[theorem]{Proposition}
\theoremstyle{definition}
\newtheorem{definition}[theorem]{Definition}

\theoremstyle{remark}
\newtheorem{remark}[theorem]{Remark}
\numberwithin{equation}{section}

\title{On the quadratic twist of elliptic curves with full $2$-torsion}
\author{Zhangjie Wang}
\address{School of Mathematics and statistics, Shaanxi Normal University, Xi'an 710119, China}
\email{zhangjiewang@snnu.edu.cn}
\author{Shenxing Zhang}
\address{School of Mathematics, Hefei University of Technology, Hefei, Anhui 230000, China}
\email{zhangshenxing@hfut.edu.cn}
\date{\today}

\begin{document}
\maketitle
\tableofcontents

\begin{abstract}
Let $E: y^2=x(x-a^2)(x+b^2)$ be an elliptic curve with full $2$-torsion group, where $a$ and $b$ are coprime integers and $2(a^2+b^2)$ is a square. Assume that the $2$-Selmer group of $E$ has rank two. We characterize all quadratic twists of $E$ with Mordell-Weil rank zero and $2$-primary Shafarevich-Tate groups $(\mathbb Z/2\mathbb Z)^2$, under certain conditions. We also obtain a distribution result of these elliptic curves.
\end{abstract}

\section{Introduction}

A positive integer $n$ is called {\em a congruent number} if and only if it is the area of a right rational triangle.
This is equivalently to say, the congruent elliptic curve $y^2=x^3-n^2x$ has positive Mordell-Weil rank.
The classical descent method provides a way to construct non-congruent numbers, see~\cite{Feng1996, LiTian2000, OuyangZhang2014, OuyangZhang2015}.
In \cite{Wang2016}, the first author used Cassels pairing to characterize all congruent elliptic curves $y^2=x^3-n^2x$ with Mordell-Weil rank zero and second minimal $2$-primary Shafarevich-Tate group, where all prime divisors of $n$ are congruent to $1$ modulo $4$.

Fujiware in \cite{Fujiwara1998} defined the generalized concept, a $\theta$-congruent number by considering rational triangles with an angle $\theta$.
For such a triangle, $\cos\theta=s/r$ is rational number, $s,r\in\BZ,r>0,\gcd(r,s)=1$.
A positive integer $n$ is called {\em $\theta$-congruent} if $n\sqrt{r^2-s^2}$ is the area of a rational triangle with an
angle $\theta$.
This is equivalently to say, the elliptic curve 
\[y^2=x(x+(r+s)n)(x-(r-s)n)\] has positive Mordell-Weil rank.
The goal of this paper is to generalize the result in \cite{Wang2016} to non-$\theta$-congruent case.
More precisely, we will show that for certain $n$, $n$ is non-congruent with second minimal $2$-primary Shafarevich-Tate group, if and only if $n$ is non-$\theta$-congruent with second minimal $2$-primary Shafarevich-Tate group, where both of $\sqrt2\sin(\theta/2)$ and $\sqrt2\cos(\theta/2)$ are rational.

Let $(a,b,c)$ be a primitive triple of positive integers such that $a^2+b^2=2c^2$.
By elementary number theory, this is equivalent to say,
\[a=|\alpha^2-2\alpha\beta-\beta^2|,\quad b=|\alpha^2+2\alpha\beta-\beta^2|,\quad c=\alpha^2+\beta^2\]
for some coprime integers $\alpha,\beta$ with different parities.
Denote by
\[E: y^2=x(x-a^2)(x+b^2)\]
an elliptic curve with full $2$-torsion group, and
\[E^{(n)}: y^2=x(x-a^2n)(x+b^2n)\]
a quadratic twist of $E$, where $n$ is a positive square-free integer.
If 
\[\sin\Bigl(\frac\theta2\Bigr)=\frac a{\sqrt 2c},\quad
\cos\Bigl(\frac\theta2\Bigr)=\frac b{\sqrt 2c}
\quad\text{and}\quad
\tan\Bigl(\frac\theta2\Bigr)=\frac ab,\]
then $\cos\theta=\frac{b^2-a^2}{b^2+a^2}$.
Therefore, $n$ is non-$\theta$-congruent if and only if $\rank_\BZ E^{(n)}=0$.
%

\subsection{Rank zero twists}

When $n>1$, denote by $\CA$ the ideal class group of $K=\BQ(\sqrt{-n})$ and
\[h_{2^m}(n):=\dim_{\BF_2} \CA^{2^{m-1}}/\CA^{2^m}\]
its $2^m$-rank for a positive integer $m$.
Denote by $\Sel_2\bigl(E^{(n)}/\BQ\bigr)$ the $2$-Selmer group of $E^{(n)}$ over $\BQ$.

\begin{theorem}[=Theorems~\ref{thm:main_A} and \ref{thm:main_B}]
\label{thm:main}
Assume that $\Sel_2(E/\BQ)\cong(\BZ/2\BZ)^2$.
Let $n\equiv1\bmod8$ be a positive square-free integer coprime to $abc$ where each prime factor of $n$ is a quadratic residue modulo every prime factor of $abc$.

\textup{(A)} If all prime factors of $n$ are congruent to $\pm1$ modulo $8$, then the following are equivalent:
\begin{enumerate}
\item $\rank_\BZ E^{(n)}(\BQ)=0$ and $\Sha(E^{(n)}/\BQ)[2^\infty]\cong(\BZ/2\BZ)^2$;
\item $h_4(n)=1$ and $h_8(n)=0$.
\end{enumerate}

\textup{(B)} If all prime factors of $n$ are congruent to $1$ modulo $4$, then the following are equivalent:
\begin{enumerate}
\item $\rank_\BZ E^{(n)}(\BQ)=0$ and $\Sha(E^{(n)}/\BQ)[2^\infty]\cong(\BZ/2\BZ)^2$;
\item $h_4(n)=1$ and $h_8(n)\equiv\frac{d-1}4\bmod2$.
\end{enumerate}
Here $d$ is the odd part of $d_0\mid 2n$ such that the ideal class $[(d_0,\sqrt{-n})]$ is the non-trivial element in $\CA[2]\cap\CA^2$.
\end{theorem}

\begin{remark}
(1) When $(a,b)=(1,1),(7,23),(23,47),(119,167),(167,223),(287,359)$, we have $\Sel_2(E/\BQ)\cong(\BZ/2\BZ)^2$.

(2) In Theorem~\ref{thm:main}(B), if $h_4(n)=1$, then the non-trivial element in $\CA[2]\cap\CA^2$ is $[(d_0,\sqrt{-n})]$ for some positive divisor $d_0$ of $2n$.
If $d_0'$ is another positive divisor of $2n$ such that $[(d_0,\sqrt{-n})]=[(d_0',\sqrt{-n})]$, then $d_0d_0'=n$ or $4n$.
See \S\ref{Gauss genus theory}.
\end{remark}

We will first show that $E^{(n)}_\tor(\BQ)\cong(\BZ/2\BZ)^2$ in \S\ref{Torsion subgroup}.
In \S\ref{2-descent method}, we will study the local solvability of homogeneous spaces and then express the $2$-Selmer group as the kernel of the generalized Monsky matrix $\CM_n$.
Then we will give the proof of Theorem~\ref{thm:main} in \S\ref{Second minimal Shafarevich-Tate group}.
The strategy is similar to \cite{Wang2016}.

\subsection{Distribution}

We will also study the distribution of $n$ satisfying the equivalent condition in Theorem~\ref{thm:main}(B), which generalizes the result in \cite{Wang2017}.
Denote by
\begin{itemize}
\item $C_k(x)$ the set of positive square-free integers $n\le x$ with exactly $k$ prime factors;
\item $\msq_k(x)$ the set of $n\in C_k(x)$ coprime to $abc$ such that each prime factor of $n\equiv1\bmod8$ is a quadratic residue modulo every prime factor of $abc$ and congruent to $1$ modulo $4$;
\item $\msp_k(x)$ the set of all $n\in\msq_k(x)$ such that Theorem~\ref{thm:main}(B)(2) holds.
\end{itemize}
We will use the standard symbols in analytic number theory: ''$\sim, \ll, O(\cdot), o(\cdot), \Li(x)$'', which can be found in \cite{IrelandRosen1990}.
The equidistribution property of Legendre symbols in \cite{Rhoades2009} implies
\begin{equation}\label{eq:number-C-k}
\#C_k(x)\sim\frac{x(\log \log x)^{k-1}}{(k-1)!\log x}.
\end{equation}

\begin{theorem}\label{thm:dist}
Assume that $\Sel_2(E/\BQ)\cong(\BZ/2\BZ)^2$.
Then
\[\#\msp_k(x)\sim 2^{-k\ell-k-2}\left(u_k+ (2^{-1}-2^{-k})u_{k-1}\right) \cdot \#C_k(x),\]
where $\ell$ is the number of different prime factors of $abc$ and
\[u_k:=\prod_{1\le i\le k/2}(1-2^{1-2i}).\]
\end{theorem}

We will use the method in \cite{CremonaOdoni1989} to show the equidistribution property of residue symbols in \S~\ref{Equidistribution of residue symbols} and then use this to prove Theorem~\ref{thm:dist} in \S~\ref{Distribution result}.

\begin{remark}
If one want to study the distribution of $n$ satisfying the equivalent condition in Theorem~\ref{thm:main}(A), one need a characterization of $h_8(n)$ in terms of residue symbols similar to \cite{JungYue2011}.
\end{remark}

\subsection{Notations}

We will not list the notations appeared above.
\begin{itemize}
\item $n=p_1\cdots p_k$ the prime decomposition of $n$.
\item $abc=q_1^{t_1}\cdots q_\ell^{t_\ell}$ the prime decomposition of $abc$.
\item $\gcd(m_1,\dots,m_t)$ the greatest common divisor of integers $m_1,\dots,m_t$.
\item $\Sel_2'\bigl(E^{(n)}\bigr)=\Sel_2\bigl(E^{(n)}\bigr)/E^{(n)}(\BQ)[2]$ the pure $2$-Selmer group of $E^{(n)}$, see \eqref{eq:pure2selmer}.
\item $D_\Lambda$ the homogeneous space associated to a rational triple $(d_1,d_2,d_3)$, see~\eqref{eq:D-Lambda}.
\item $(\alpha,\beta)_v$ the Hilbert symbol, $\alpha,\beta\in\BQ_v^\times$.
\item $[\alpha,\beta]_v$ the additive Hilbert symbol, i.e., the image of $(\alpha,\beta)_v$ under the isomorphism $\set{\pm1}\simto\BF_2$.
\item $\leg\alpha\beta=\prod_{p\mid\beta}(\alpha,\beta)_p$ the Jacobi symbol with $p\mid \beta$ counted with multiplicity, where $\gcd(\alpha,\beta)=1$ and $\beta>0$.
\item $\aleg\alpha\beta$ the additive Jacobi symbol, i.e., the image of $\leg\alpha\beta$ under the isomorphism $\set{\pm1}\simto\BF_2$.
\item $\CD(K)$ the set of positive square-free divisors of $2n$.
\item ${\bf0}=(0,\dots,0)^\rmT$ and ${\bf1}=(1,\dots,1)^\rmT$.
\item $\bfI$ the identity matrix and $\bfO$ the zero matrix.
\item $\bfA=\bfA_n$ a matrix associated to $n$, see~\eqref{eq:A-n}.
\item $\bfR_n$ the R\'edei matrix of $K=\BQ(\sqrt{-n})$, see \eqref{eq:Redei-matrix}.
\item $\bfD_u=\diag\Bigl\{\aleg u{p_1},\dots,\aleg u{p_k}\Bigr\}$.
\item $\bfb_u=\bfD_u{\bf1}=\Bigl(\aleg u{p_1},\dots,\aleg u{p_k}\Bigr)$.
\item $\bfM_n$ the Monsky matrix associated to $n$, see~\eqref{eq:bfM-n}.
\item $\CM_n$ the generalized Monsky matrix associated to $E^{(n)}$, see~\eqref{eq:CM-n}.
\item $I=\sqrt{-1}$.
\item $\CP$ the set of primary primes of $\BZ[I]$ with positive imaginary part.
\item $\leg\alpha\lambda_2$ the quadratic residue symbol over $\BZ[I]$, see~\eqref{eq:quadratic-residue-symbol}.
\item $\leg\alpha\lambda_4$ the quartic residue symbol over $\BZ[I]$, see~\eqref{eq:quartic-residue-symbol}.
\item $\leg{a}{d}_4:=\leg{a}{\lambda}_4$ the rational quartic residue symbol, see~\eqref{eq:rational-quartic-residue-symbol}.
\item $\Lambda(\fa)$ the Mangoldt function, see~\eqref{eq:Mangoldt-function}.
\item $\chi_0$ the trivial character modulo a given integral ideal, see \S~\ref{Analytic results}.
\item $\psi(x,\chi)=\sum_{\bfN\fa\le x}\chi(\fa)\Lambda(\fa)$, see~\eqref{eq:psi-x-chi}.
\item $C_k(x,\alpha,\bfB),C_k'(x,\alpha,\bfB),T_k(x),T_k'(x)$ sets associated to $x,\alpha,\bfB$, see \S~\ref{Equidistribution of residue symbols}.
\item $\binom k2=k(k-1)/2$ the binomial coefficient.
\end{itemize}

\section{Preliminaries}
\label{Preliminaries}

\subsection{Gauss genus theory}
\label{Gauss genus theory}

In this subsection, we will recall Gauss genus theory, which can be found in \cite[\S~3]{Wang2016} for details.
For our purpose, assume that $n=p_1\cdots p_k\equiv1\bmod 4$.
Denote by $\CA$ the ideal class group of $K=\BQ(\sqrt{-n})$.
Denote by $\CD(K)$ the set of positive square-free divisors of $2n$.
The classical Gauss genus theory tells that
\[\CA[2]=\set{[(d,\sqrt{-n})]: d\in\CD(K)}
\quad\text{and}\quad
h_2(n)=\dim_{\BF_2}\CA[2]=t-1.\]

Denote by $p_{k+1}=2$ and define the R\'edei matrix
\begin{equation}\label{eq:Redei-matrix}
\bfR_n=\bigl([p_j,-n]_{p_i}\bigr)_{i,j}\in M_{k\times (k+1)}(\BF_2).
\end{equation}

\begin{proposition}[{\cite{Redei1934}}]
\label{pro:4rank}
We have
\[\begin{array}{rcccl}
\Ker\bfR_n&\stackrel{\sim}{\longleftarrow}&\CD(K)\cap\bfN_{K/\BQ} K^\times
&\lra&\CA[2]\cap \CA^2\\
\bigl(v_{p_1}(d),\dots,v_{p_{k+1}}(d)\bigr)&\longmapsfrom&d&\longmapsto&[(d,\sqrt{-n})],
\end{array}\]
where the second arrow is a two-to-one onto homomorphism with kernel $\set{1,n}$.
Hence $h_4(n)=k-\rank\bfR_n$.
\end{proposition}

\begin{proposition}[{\cite[Proposition~3.6]{Wang2016}}]\label{pro:8rank}
For any $2^rd\in \CD(K)\cap\bfN_{K/\BQ}K^\times$ with odd $d$, let $(\alpha,\beta,\gamma)$ be a primitive triple of positive integers satisfying
\[d\alpha^2+\frac{n}{d}\beta^2=2^r\gamma^2.\]
Then $[(2^rd,\sqrt{-n})]\in\CA^4$ if and only if
\[\bfb_\gamma=\Bigl(\aleg\gamma{p_1},\dots,\aleg\gamma{p_k}\Bigr)^\rmT\in\Im\bfR_n.\]
\end{proposition}

\subsection{Torsion subgroup}
\label{Torsion subgroup}

\begin{proposition}\label{pro:torsion-subgroup}
For any positive square-free integer $n$, $E^{(n)}_\tor(\BQ)\cong(\BZ/2\BZ)^2$.
\end{proposition}

\begin{lemma}[{\cite{Ono1996}}]\label{lem:ono}
Let $\CE: y^2=x(x-a)(x+b)$ be an elliptic curve with $a,b\in\BZ$.
\begin{enumerate}
\item $\CE(\BQ)$ has a point of order $4$ if and only if one of the three pairs $(-a,b), (a,a+b)$ and $(-b,-a-b)$ consists of squares of integers.
\item $\CE(\BQ)$ has a point of order $3$ if and only if there exist integers $d,u,v$ such that $\gcd(u,v)=1$, $d^2u^3(u+2v)=-a, d^2v^3(v+2u)=b$ and $u/v\not\in\set{-2,-1/2,-1,1,0}$.
\end{enumerate}
\end{lemma}

\begin{proof}[Proof of Proposition~\ref{pro:torsion-subgroup}]
Since $E^{(n)}$ has full rational $2$-torsion, $E^{(n)}_\tor(\BQ)$ contains a subgroup isomorphic to $(\BZ/2\BZ)^2$.
By Mazur's classification theorem \cite{Mazur1977,Mazur1978}, one have
\[E^{(n)}_\tor(\BQ)\cong\BZ/2\BZ\oplus \BZ/2N\BZ\]
for some $N\in\set{1,2,3,4}$.
We only need to show that $E^{(n)}(\BQ)$ contains no point of order $4$ or $3$.

Since the three pairs in Lemma~\ref{lem:ono}(1) are $(-a^2n, b^2n), (a^2n, 2c^2n)$ and $(-b^2n, -2c^2n)$, $E^{(n)}(\BQ)$ contains no point of order $4$.

Assume that there are integers $d,u,v$ such that $\gcd(u,v)=1$,
\[d^2u^3(u+2v)=-a^2n \quad\text{and}\quad d^2v^3(v+2u)=b^2n.\]
Clearly, $d^2=1$ and $n=\gcd(u+2v,v+2u)=\gcd(3,u-v)=1$ or $3$.
Since $a$ and $b$ are odd, so is $u,v$.
We may assume that $v>0$, then $u<0$.
Since $n\mid (u+2v,v+2u)$, we may write $v=\alpha^2, u=-\beta^2$.
Then $(\alpha^2-2\beta^2)/n$ and $(2\alpha^2-\beta^2)/n$ are squares, which is impossible by modulo $8$.
Hence $E^{(n)}(\BQ)$ contains no point of order $3$ by Lemma~\ref{lem:ono}(2).
\end{proof}

\subsection{Cassels pairing}

As shown in \cite{Cassels1998}, the $2$-Selmer group $\Sel_2\bigl(E^{(n)}\bigr)$ can be identified with
\[\set{\Lambda=(d_1,d_2,d_3)\in \bigl(\BQ^\times/\BQ^{\times2}\bigr)^3: D_\Lambda(\BA_\BQ)\neq\emptyset,d_1d_2d_3\equiv1\bmod\BQ^{\times2}},\]
where $D_\Lambda$ is a genus one curve defined by
\begin{equation}\label{eq:D-Lambda}
\begin{cases}
	H_1:& -b^2nt^2+d_2u_2^2-d_3u_3^2=0, \\
	H_2:& -a^2nt^2+d_3u_3^2-d_1u_1^2=0, \\
	H_3:& 2c^2nt^2+d_1u_1^2-d_2u_2^2=0.
\end{cases}
\end{equation}
Under this identification, the points $O,(a^2n,0),(-b^2n,0),(0,0)$ and non-torsion $(x,y)\in E^{(n)}(\BQ)$ correspond to
\begin{equation}\label{eq:torsion-homogeneous}
(1,1,1),\ (2,2n,n),\ (-2n,2,-n),\ (-n,n,-1)
\end{equation}
and $(x-a^2n,x+b^2n,x)$ respectively.

Cassels in \cite{Cassels1998} defined a skew-symmetric bilinear pairing $\pair{-,-}$ on the $\BF_2$-vector space
\begin{equation}\label{eq:pure2selmer}
\Sel_2'\bigl(E^{(n)}\bigr):=\Sel_2\bigl(E^{(n)}\bigr)/E^{(n)}(\BQ)[2].
\end{equation}
We will write it additively.
For any $\Lambda\in\Sel_2\bigl(E^{(n)}\bigr)$, choose $P=(P_v)\in D_\Lambda(\BA_\BQ)$.
Since $H_i$ is locally solvable everywhere, there exists $Q_i\in H_i(\BQ)$ by Hasse-Minkowski principle.
Let $L_i$ be a linear form in three variables such that $L_i=0$ defines the tangent plane of $H_i$ at $Q_i$.
Then for any $\Lambda'=(d_1',d_2',d_3')\in \Sel_2\bigl(E^{(n)}\bigr)$, define
\[\pair{\Lambda,\Lambda'}=\sum_v \pair{\Lambda,\Lambda'}_v\in\BF_2,\quad\text{where}\quad \pair{\Lambda,\Lambda'}_v=\sum_{i=1}^3 \bigl[L_i(P_v), d_i'\bigr]_v.\]
This pairing is independent of the choice of $P$ and $Q_i$, and is trivial on $E^{(n)}(\BQ)[2]$.

\begin{lemma}[{\cite[Lemma 7.2]{Cassels1998}}]\label{lem:cassels}
The local Cassels pairing $\pair{\Lambda,\Lambda'}_p=0$ if
\begin{itemize}
\item $p\nmid 2\infty$,
\item the coefficients of $H_i$ and $L_i$ are all integral at $p$ for $i=1,2,3$, and
\item modulo $D_\Lambda$ and $L_i$ by $p$, they define a curve of genus $1$ over $\BF_p$ together with tangents to it.
\end{itemize}
\end{lemma}

\begin{lemma}\label{lem:non-deg}
The following are equivalent:
\begin{enumerate}
\item $\rank_\BZ E^{(n)}(\BQ)=0$ and $\Sha(E^{(n)}/\BQ)[2^\infty]\cong (\BZ/2\BZ)^{2t}$;
\item $\Sel_2'\bigl(E^{(n)}\bigr)\cong (\BZ/2\BZ)^{2t}$ and the Cassels pairing on $\Sel_2'\bigl(E^{(n)}\bigr)$ is non-degenerate.
\end{enumerate}
\end{lemma}
\begin{proof}
Note that $E^{(n)}(\BQ)[2]=(\BZ/2\BZ)^2$ by Proposition~\ref{pro:torsion-subgroup}.
The proof is similar to \cite[p.~2157]{Wang2016}.
\end{proof}

By this lemma, the proof of our main result can be reduced to the calculations of the $2$-Selmer group and the Cassels pairing on it.

\section{$2$-descent method}
\label{2-descent method}

In this section, we will study the local solvability of homogeneous spaces and then express the $2$-Selmer group as the kernel space of a matrix defined over $\BF_2$.

\subsection{Homogeneous spaces}

\begin{lemma}\label{lem:local-solv}
Let $n$ be a positive square-free integer prime to $2abc$ and $\Lambda=(d_1,d_2,d_3)$, where $d_1,d_2,d_3$ are square-free integers.
\begin{enumerate}
\item If $p\nmid 2abcn$, then $D_\Lambda(\BQ_p)\neq\emptyset$ if and only if $p\nmid d_1d_2d_3$.
\item If $D_\Lambda(\BQ_2)\neq\emptyset$, then $d_1$ and $d_2$ have the same parity.
\item If both of $d_1$ and $d_2$ are odd, then $D_\Lambda(\BQ_2)\neq\emptyset$ if and only if either $4\mid d_1-1, 8\mid d_1-d_2$ or $4\mid d_1+n, 8\mid d_1-d_2+2n$.
\item $D_\Lambda(\BR)\neq\emptyset$ if and only if $d_2>0$.
\end{enumerate}
\end{lemma}

\begin{proof}
Certainly, $\gcd(d_1,d_2,d_3)=1$.
Since we are dealing with homogeneous equations, we may assume that $u_1, u_2, u_3$ and $t$ are $p$-adic integers and at least one of them is a $p$-adic unit.

(1) By classical descent theory, see \cite[Theorem~X.1.1, Corollary~X.4.4]{Silverman2009}.
%

(2) Suppose that $D_\Lambda(\BQ_2)\neq\emptyset$.
If $2\mid d_1, 2\nmid d_2$, then $2\mid d_3$.
We have $2\mid u_2$ by $H_3$ and $2\mid t$ by $H_1$.
Then $2\mid u_3$ by $H_1$ and $2\mid u_1$ by $H_2$, which is impossible.
The case $2\nmid d_1, 2\mid d_2$ is similar.
Hence $d_1$ and $d_2$ have the same parity.

(3) If $D_\Lambda(\BQ_2)\neq\emptyset$, then both of $u_1,u_2$ are odd by $H_3$ and exactly one of $t$ and $u_3$ is even by $H_2$.
If $t$ is even and $u_3$ is odd, then $4\mid d_1-d_3, 8\mid d_1-d_2$ by $H_2\bmod 4$ and $H_3\bmod 8$.
Note that if $8\mid d_1-d_2$, then $d_3\equiv d_1d_2\equiv 1\bmod8$.
If $t$ is odd and $u_3$ is even, then $4\mid d_1+n, 8\mid d_1-d_2+2n$ by $H_2\bmod 4$ and $H_3\bmod 8$.

Conversely, if $4\mid d_1-1, 8\mid d_1-d_2$, then $d_3\equiv d_1d_2\equiv 1\bmod 8$. Take
\begin{itemize}
\item $t=0, u_1=\sqrt{1/d_1}, u_2=\sqrt{1/d_2}, u_3=\sqrt{1/d_3}$ if $8\mid d_1-1$;
\item $t=2, u_1=1, u_2=\sqrt{(d_1+8c^2n)/d_2}, u_3=\sqrt{(d_1+4a^2n)/d_3}$ if $8\mid d_1-5$.
\end{itemize}
If $4\mid d_1+n, 8\mid d_1-d_2+2n$, take
\begin{itemize}
\item $t=1, u_1=\sqrt{-a^2n/d_1}, u_2=\sqrt{b^2n/d_2}, u_3=0$ if $8\mid d_1+n$;
\item $t=1, u_1=\sqrt{(4d_3-a^2n)/d_1}, u_2=\sqrt{(4d_3+b^2n)/d_2}, u_3=2$ if $8\mid d_1+n+4$.
\end{itemize}

(4) Suppose that $D_\Lambda(\BR)\neq\emptyset$.
If $d_2<0$, then $d_3<0$ by $H_1$.
Thus $d_1>0$ by $d_1d_2d_3\in\BQ^{\times2}$ and $d_1<0$ by $H_2$, which is impossible.
Hence $d_2>0$.
Another direction is trivial.
\end{proof}

Assume that $n$ is a positive square-free integer prime to $2abc$.
By Lemma~\ref{lem:local-solv} and \eqref{eq:torsion-homogeneous}, any element of the pure $2$-Selmer group $\Sel_2'\bigl(E^{(n)}\bigr)$ has a unique representative $\Lambda=(d_1,d_2,d_3)$, where $d_1,d_2,d_3$ are positive square-free integers dividing $nabc$.
In the rest part of this article, $\Lambda$ is always assumed to be in this form and we will write $\Lambda=(d_1,d_2,d_3)\in\Sel_2'\bigl(E^{(n)}\bigr)$ for simplicity.

\begin{lemma}\label{lem:local-solv-n}
Let $n$ be a positive square-free integer prime to $2abc$ and $\Lambda=(d_1,d_2,d_3)$.
Let $p$ be a prime factor of $n$.
Then $D_\Lambda(\BQ_p)\neq\emptyset$ if and only if
\begin{itemize}
\item $\leg{d_1}p=\leg{d_2}p=1$, if $p\nmid d_1,p\nmid d_2$;
\item $\leg{2d_1}p=\leg{2n/d_2}p=1$, if $p\nmid d_1,p\mid d_2$;
\item $\leg{-2n/d_1}p=\leg{2d_2}p=1$, if $p\mid d_1,p\nmid d_2$;
\item $\leg{-n/d_1}p=\leg{n/d_2}p=1$, if $p\mid d_1,p\mid d_2$.
\end{itemize}
\end{lemma}

\begin{proof}
Assume that $p\nmid d_1d_2$, then $p\nmid d_3$.
If $D_\Lambda(\BQ_p)\neq\emptyset$, then $\leg{d_2d_3}p=\leg{d_1d_3}p=1$ by $H_2$ and $H_3$.
That's to say, $\leg{d_1}p=\leg{d_2}p=1$.
Conversely, if $\leg{d_1}p=\leg{d_2}p=1$, then $\bigl(0,\sqrt{1/d_1},\sqrt{1/d_2},\sqrt{1/d_3}\bigr)\in D_\Lambda(\BQ_p)$. The rest cases can be proved similarly as in the congruent elliptic curve case, see \cite[Appendix]{HeathBrown1994}.
\end{proof}

\begin{lemma}\label{lem:local-solv-abc}
Let $n$ be a positive square-free integer prime to $2abc$ and $\Lambda=(d_1,d_2,d_3)$.
Let $p$ be a prime factor of $abc$.
\begin{enumerate}
\item If $p\mid a$, then $D_\Lambda(\BQ_p)\neq\emptyset$ if and only if one of the following cases holds:
	\begin{itemize}
	\item $p\nmid d_2, p\nmid d_1, \leg{d_2}p=1$;
	\item $p\nmid d_2, p\mid d_1, \leg{d_2}p=\leg{n}p=1$.
	\end{itemize}
\item If $p\mid b$, then $D_\Lambda(\BQ_p)\neq\emptyset$ if and only if one of the following cases holds:
	\begin{itemize}
	\item $p\nmid d_1, p\nmid d_2, \leg{d_1}p=1$;
	\item $p\nmid d_1, p\mid d_2, \leg{d_1}p=\leg{-n}p=1$.
	\end{itemize}
\item If $p\mid c$, then $D_\Lambda(\BQ_p)\neq\emptyset$ if and only if one of the following cases holds:
	\begin{itemize}
	\item $p\nmid d_3, p\nmid d_1, \leg{d_3}p=1$;
	\item $p\nmid d_3, p\mid d_1, \leg{d_3}p=\leg{n}p=1$.
	\end{itemize}
\end{enumerate}
\end{lemma}

\begin{proof}
Let $p$ be a prime factor of $a$.

Suppose that $D_\Lambda(\BQ_p)\neq\emptyset$.
If $p\mid d_2$, then $p$ divides exactly one of $d_1$ and $d_3$. We may assume that $p\mid d_1$ and $p\nmid d_3$.
Then $p$ divides $u_3,t$ by $H_2,H_3$ and then $u_2,u_1$ by $H_1,H_2$.
So $p\mid\gcd(t,u_1,u_2,u_3)$, which will cause a contradiction.
Hence $p\nmid d_2$.

Suppose that $p\nmid d_1, p\nmid d_3$.
If $D_\Lambda(\BQ_p)\neq\emptyset$, then $\leg{d_1d_3}p=\leg{d_2}p=1$ by $H_2$.
Conversely, if $\leg{d_2}p=1$, then we may take
\[\begin{split}
u_1&=d_2/\gcd(d_1,d_2),\\
u_3^2&=d_2+a^2nt^2/d_3\equiv d_2\bmod p,\\
u_2^2&=d_3+2c^2nt^2/d_2,
\end{split}\]
where $t\in\BZ_p$ such that $d_3+2c^2nt^2/d_2$ is a square in $\BZ_p$.
In fact, if $-2nd_3$ is quadratic residue modulo $p$, then we may take $t=\sqrt{-\frac{d_2d_3}{2c^2n}}$ and $u_2=0$; if $-2nd_1$ is not a quadratic residue modulo $p$, then there exists $t\in\set{0,1,\dots,(p-1)/2}$ such that $d_3+2c^2nt^2/d_2\bmod p$ is a nonzero square.
Hence $D_\Lambda(\BQ_p)$ is non-empty.

Suppose that $p\mid d_1, p\mid d_3$.
If $D_\Lambda(\BQ_p)\neq\emptyset$, then $\leg{d_2n}p=1$ by $H_1$ and $\leg{d_2}p=1$ by $H_2$.
Conversely, if $\leg{d_2}p=\leg{n}p=1$, then we may take $t=1$ and
\[\begin{split}
u_1&=d_2/\gcd(d_1,d_2),\\
u_3^2&=d_2+a^2n/d_3\equiv d_2\bmod p,\\
u_2^2&=d_3+2c^2n/d_2\equiv b^2n/d_2\bmod p.
\end{split}\]
Hence $D_\Lambda(\BQ_p)$ is non-empty.

The rest cases can be proved similarly.
\end{proof}

\begin{lemma}\label{lem:local-solv-2}
Let $n$ be a positive square-free integer prime to $2abc$ and $\Lambda=(d_1,d_2,d_3)$.
If $D_\Lambda(\BQ_v)\neq\emptyset$ for all places $v\neq 2$, then $D_\Lambda(\BQ_2)$ is also non-empty.
\end{lemma}

\begin{proof}
Since $D_\Lambda(\BQ_v)\neq\emptyset$ for all places $v\neq 2$,
each $H_i$ is locally solvable at $v\neq 2$.
By the product formula of Hilbert symbols, $H_i$ is locally solvable at $2$.
In other words,
\[[nd_2,d_2d_3]_2=[-nd_1,d_3d_1]_2=[2nd_2,d_1d_2]_2=0.\]
Then $[nd_2,d_1]_2=[-nd_1,d_2]_2=0$.
\begin{itemize}
\item If $d_1\equiv d_2\bmod 4$, then $[-n,d_1]_2=[n,d_2]_2=[2,d_1d_2]_2=0$, which forces $4\mid d_1-1$ and $8\mid d_1-d_2$.
\item If $d_1\equiv -d_2\bmod 4$, then $[n,d_1]_2=[-n,-d_1]_2=0$ and $n\equiv -d_1\equiv d_2\bmod 4$. Since $[2,d_1d_2]_2=[2nd_2,d_1d_2]_2=0$, we have $d_1d_2\equiv -1\bmod 8$.
In other words, $4\mid d_1+n$ and $8\mid d_1-d_2+2n$.
\end{itemize}
Hence $D_\Lambda(\BQ_2)\neq\emptyset$ by Lemma~\ref{lem:local-solv-2}(3).
\end{proof}

\subsection{Matrix representation}
\label{Matrix representation}

By the results in the previous subsection, we can express the pure $2$-Selmer group $\Sel_2'\bigl(E^{(n)}\bigr)$ as the kernel of a matrix.
For our purpose, we assume that $n$ is prime to $abc$ and each prime factor of $n$ is a quadratic residue modulo every prime factor of $abc$.

Denote by $n=p_1\cdots p_k$ and
\begin{equation}\label{cd-factor-abc}
a=q_1^{t_1}\cdots q_{\ell_1}^{t_{\ell_1}},\quad
b=q_{\ell_1+1}^{t_{\ell_1+1}}\cdots q_{\ell_2}^{t_{\ell_2}},\quad
c=q_{\ell_2+1}^{t_{\ell_2+1}}\cdots q_\ell^{t_\ell}
\end{equation}
the prime decompositions respectively, where all $t_i>0$ and $0\le \ell_1\le \ell_2\le \ell$.
Let $\Lambda=(d_1,d_2,d_3)\in\Sel_2'\bigl(E^{(n)}\bigr)$ where $d_1,d_2,d_3$ are positive square-free integers dividing $nabc$.
By Lemma~\ref{lem:local-solv-abc}, we have $\gcd(a,d_2)=\gcd(b,d_1)=\gcd(c,d_3)=1$.
In other words, $d_1\mid nac, d_2\mid nbc$ and $d_3\mid nab$.
So we may write
\[\begin{split}
d_1&=p_1^{x_1}\cdots p_k^{x_k}\cdot
	q_1^{z_1}\cdots q_{\ell_1}^{z_{\ell_1}}\cdot
	q_{\ell_2+1}^{z_{\ell_2+1}}\cdots q_\ell^{z_\ell},\\
d_2&=p_1^{y_1}\cdots p_k^{y_k}\cdot
	q_{\ell_1+1}^{z_{\ell_1+1}}\cdots q_{\ell_2}^{z_{\ell_2}}\cdot
	q_{\ell_2+1}^{z_{\ell_2+1}}\cdots q_\ell^{z_\ell},\\
d_3&\equiv p_1^{x_1+y_1}\cdots p_k^{x_k+y_k}\cdot
	q_1^{z_1}\cdots q_{\ell_1}^{z_{\ell_1}}\cdot
	q_{\ell_1+1}^{z_{\ell_1+1}}\cdots q_{\ell_2}^{z_{\ell_2}}\bmod\BQ^{\times2}.
\end{split}\]
Denote by
\[\bfx=(x_1,\dots,x_k)^\rmT,\quad\bfy=(y_1,\dots,y_k)^\rmT\in\BF_2^k,\]
and
\[\bfz=(z_1,\dots,z_{\ell_1},z_{\ell_1+1},\dots,z_{\ell_2},z_{\ell_2+1},\dots,z_\ell)^\rmT\in\BF_2^\ell.\]

Denote by
\[\begin{pmatrix}
	\bfF_1&	\bfF_2 &	\bfF_3 \\
	\bfF_4&	\bfF_5 &	\bfF_6 \\
	\bfF_7&	\bfF_8 &	\bfF_9 \\
\end{pmatrix}=\bigl([q_j,q_i]_{q_i}\bigr)_{i,j}\in M_\ell(\BF_2),\]
where $\bfF_1\in M_{\ell_1}(\BF_2)$ and $\bfF_5\in M_{\ell_2-\ell_1}(\BF_2)$.
Denote by
\[\CM_1=\begin{pmatrix}
					&	\bfF_2	&	\bfF_3	\\
	\bfF_4	& 				&	\bfF_6	\\
	\bfF_7	&	\bfF_8	&	\\
					&	\Delta	&
\end{pmatrix}\in M_{(\ell+\ell_2-\ell_1)\times\ell}(\BF_2),\]
where
\[\Delta=\diag\Bigl(\aleg{-1}{q_{\ell_1+1}},\cdots,\aleg{-1}{q_{\ell_2}}\Bigr).\]

\begin{lemma}\label{lem:Sel-E}
Notations as above.
The map $(d_1,d_2,d_3)\mapsto\bfz$ induces an isomorphism
\[\Sel_2'(E)\simto \Ker \CM_1.\]
\end{lemma}
\begin{proof}
In the language of linear algebra, Lemma~\ref{lem:local-solv-abc} tells that
\begin{enumerate}
\item $(\bfO,\bfF_2,\bfF_3)\bfz={\bf0}$;
\item $(\bfF_4,\bfO,\bfF_6)\bfz={\bf0}$ and $\Delta (z_{\ell_1+1},\dots,z_{\ell_2})^\rmT={\bf0}$;
\item $(\bfF_7,\bfF_8,\bfO)\bfz={\bf0}$.
\end{enumerate}
The result then follows from Lemmas~\ref{lem:local-solv}(4) and \ref{lem:local-solv-2} by noting that $n=1$.
\end{proof}

Denote by
\[\bfD_u=\diag\set{\aleg u{p_1},\cdots,\aleg u{p_k}}\in M_k(\BF_2),\]
\begin{equation}\label{eq:A-n}
\bfA=\bfA_n=\bigl([p_j,-n]_{p_i}\bigr)_{i,j}\in M_k(\BF_2)
\end{equation}
and
\[(\bfG_1,\bfG_2,\bfG_3)=\bigl([q_j,-n]_{p_i}\bigr)_{i,j}\in M_{k\times\ell}(\BF_2),\]
where $\bfG_1\in M_{k\times\ell_1}(\BF_2)$ and $\bfG_2\in M_{k\times(\ell_2-\ell_1)}(\BF_2)$.
Denote the Monsky matrix by
\begin{equation}\label{eq:bfM-n}
\bfM_n=\begin{pmatrix}
\bfA+\bfD_{-2}&\bfD_2\\
\bfD_2&\bfA+\bfD_2
\end{pmatrix}
\end{equation}
and the generalized Monsky matrix by
\begin{equation}\label{eq:CM-n}
\CM_n=\begin{pmatrix}
\bfM_n&\bfG\\
&\CM_1
\end{pmatrix},
\quad\text{where}\quad
\bfG=\begin{pmatrix}
			\bfG_1	&				&\bfG_3 \\
							&\bfG_2&\bfG_3
\end{pmatrix}.
\end{equation}
See \cite[Appendix]{HeathBrown1994}.

\begin{proposition}\label{pro:Sel-E-n}
Notations as above.
The map $(d_1,d_2,d_3)\mapsto \svecc{\bfx}{\bfy}{\bfz}$ induces an isomorphism
\[\Sel_2'\bigl(E^{(n)}\bigr)\simto \Ker \CM_n.\]
\end{proposition}
\begin{proof}
This follows from Lemmas~\ref{lem:local-solv}(4), \ref{lem:local-solv-n}, \ref{lem:local-solv-abc}, \ref{lem:local-solv-2} and \ref{lem:Sel-E} with $\leg nq=1$.
\end{proof}

\section{Second minimal Shafarevich-Tate group}
\label{Second minimal Shafarevich-Tate group}

In this section, we will prove Theorem~\ref{thm:main} by calculating the Cassels pairing.
Let $n=p_1\cdots p_k\equiv1\bmod8$ be a positive square-free integer prime to $abc$ where each $p_i$ is a quadratic residue modulo every prime factor of $abc$.

\subsection{Proof of Theorem~\ref{thm:main}(A)}

\begin{lemma}\label{lem:pm1-mod-8}
Assume that each $p_i\equiv\pm1\bmod8$.
Let $\bfd=(s_1,\cdots,s_k)^\rmT$ be a column vector in $\BF_2^k$ and $d=p_1^{s_1}\cdots p_k^{s_k}$.
\begin{enumerate}
\item $\bfd\in\Ker(\bfA+\bfD_{-1})$ if and only if $\bfd+\aleg{-1}{d}{\bf1}\in\Ker\bfA^\rmT$.
\item Assume that $\Sel_2(E/\BQ)\cong(\BZ/2\BZ)^2$.
Then $\dim_{\BF_2}\Sel_2'\bigl(E^{(n)}\bigr)=2$ if and only if $h_4(n)=1$.
In which case, $\Sel_2'\bigl(E^{(n)}\bigr)$ is generated by $(2,2,1)$ and $(d,1,d)$, where $\Ker(\bfA+\bfD_{-1})=\set{{\bf0},\bfd}$.
\end{enumerate}
\end{lemma}
\begin{proof}
(1) We may rearrange the ordering of the prime factors $p_i$ such that
 $p_1\equiv\cdots\equiv p_{k'}\equiv-1\bmod8$ and $p_{k'+1}\equiv\cdots\equiv p_k\equiv1\bmod 8$.
Then $\bfb_{-1}=\svec{{\bf1}'}{\bf0}$, where ${\bf1}'\in \BF_2^{k'}$.
By the quadratic reciprocity law, one can show that
\[\bfA^\rmT=\bfA+\bfD_{-1}+\bfb_{-1}\bfb_{-1}^\rmT.\]
Since $n\equiv1\bmod8$, $k'$ is even and $\bfb_{-1}^\rmT{\bf1}={\bf1}^\rmT\bfb_{-1}=\bfb_{-1}^\rmT\bfb_{-1}=k'=0\in\BF_2$.
Since $\bfA{\bf1}={\bf0}$, we have
\[\bfA^\rmT{\bf1}=(\bfA+\bfD_{-1}+\bfb_{-1}\bfb_{-1}^\rmT){\bf1}=\bfb_{-1}\]
and
\[\bfA^\rmT(\bfI+{\bf1}\bfb_{-1}^\rmT)=\bfA^\rmT+\bfb_{-1}\bfb_{-1}^\rmT=\bfA+\bfD_{-1}.\]
Hence $\bfd\in\Ker(\bfA+\bfD_{-1})$ if and only if
\[(\bfI+{\bf1}\bfb_{-1}^\rmT)\bfd=\bfd+(\bfb_{-1}^\rmT\bfd){\bf1}=\bfd+\aleg{-1}{d}{\bf1}\in\Ker\bfA^\rmT.\]

(2) Since $\dim_{\BF_2}\Sel_2'(E)=0$, we have $\Ker\CM_1=0$ by Lemma~\ref{lem:Sel-E}.
By Proposition \ref{pro:Sel-E-n}, $\dim_{\BF_2}\Sel_2'\bigl(E^{(n)}\bigr)=2$ if and only if the rank of
\[\bfM_n=\diag\set{\bfA+\bfD_{-1},\bfA}\]
is $2k-2$.
By (1), we have $\rank\bfA=\rank(\bfA+\bfD_{-1})$ and then
\[\dim_{\BF_2}\Sel_2'\bigl(E^{(n)}\bigr)=2\iff
\rank\bfA=k-1.\]
Note that the R\'edei matrix of $\BQ(\sqrt{-n})$ is $\bfR_n=(\bfA,{\bf0})$.
Then $h_4(n)=1$ if and only if $\rank\bfA=k-1$ by Proposition~\ref{pro:4rank}.

If $\rank\bfA=k-1$, then $\Ker\bfA=\set{{\bf0},{\bf1}}$.
Hence
\[\Ker\CM_n=\set{
\svecc{\bf0}{\bf0}{\bf0},
\svecc{\bf0}{\bf1}{\bf0},
\svecc{\bfd}{\bf0}{\bf0},
\svecc{\bfd}{\bf1}{\bf0}}.\]
In other words, $\Sel_2'\bigl(E^{(n)}\bigr)$ is generated by $(1,n,n)$ and $(d,1,d)$.
Conclude the result by the fact that $(1,n,n)-(2,2,1)=(2,2n,n)$  corresponds a torsion, see~\eqref{eq:torsion-homogeneous}.
\end{proof}

\begin{theorem}\label{thm:main_A}
Assume that $\Sel_2(E/\BQ)\cong(\BZ/2\BZ)^2$.
Let $n$ be a positive square-free integer prime to $abc$ where each prime factor of $n$ is a quadratic residue modulo every prime factor of $abc$.
If all prime factors of $n\equiv1\bmod8$ are congruent to $\pm1$ modulo $8$, then the following are equivalent:
\begin{enumerate}
\item $\rank_\BZ E^{(n)}(\BQ)=0$ and $\Sha(E^{(n)}/\BQ)[2^\infty]\cong(\BZ/2\BZ)^2$;
\item $h_4(n)=1$ and $h_8(n)=0$.
\end{enumerate}
\end{theorem}

\begin{proof}
By Lemma~\ref{lem:non-deg}, (1) is equivalent to say, $\Sel_2'\bigl(E^{(n)}\bigr)$ has dimension $2$ and the Cassels pairing on it is non-degenerate.
By Lemma~\ref{lem:pm1-mod-8}(2), $\dim_{\BF_2}\Sel_2'\bigl(E^{(n)}\bigr)=2$ if and only if $h_4(n)=1$.

Since all prime factors of $n$ are congruent to $\pm1$ modulo $8$, $2$ is a norm and there exists a primitive triple $(\alpha,\beta,\gamma)$ of positive integers such that
\[\alpha^2+n\beta^2=2\gamma^2.\]
It's easy to see that all of $\alpha,\beta,\gamma$ are odd.

Assume that $h_4(n)=1$.
Then by Lemma~\ref{lem:pm1-mod-8}(2), $\Sel_2'\bigl(E^{(n)}\bigr)$ is generated by $\Lambda=(2,2,1)$ and $\Lambda'=(d,1,d)$.
Recall that $D_\Lambda$ is
\[\begin{cases}
H_1:&-b^2nt^2+2u_2^2-u_3^2=0,\\
H_2:&-a^2nt^2+u_3^2-2u_1^2=0,\\
H_3:&c^2nt^2+u_1^2-u_2^2=0.
\end{cases}\]
Choose
\[\begin{aligned}
Q_1&=(\beta,b\gamma,b\alpha)\in H_1(\BQ),
&L_1&=bn\beta t-2\gamma u_2+\alpha u_3,\\
Q_3&=(0,1,1)\in H_3(\BQ),
&L_3&=u_1-u_2.
\end{aligned}\]
By Lemma~\ref{lem:cassels}, we have
\[\pair{\Lambda,\Lambda'}=\sum_{p\mid 2nabc} \bigl[L_1L_3(P_p),d\bigr]_p\]
for any $P_p\in D_\Lambda(\BQ_p)$.
Since $\leg{p_i}{q}=1$ for any prime $q\mid abc$, we have $\leg dq=1$ and $\pair{\Lambda,\Lambda'}_q=0$.

For $p\mid n$, $\alpha^2\equiv 2\gamma^2\bmod p$.
We may take $\sqrt{2}\in\BQ_p$ such that $\sqrt{2}\gamma\equiv \alpha\bmod p$.
Take $P_p=(t,u_1,u_2,u_3)=(0,1,-1,\sqrt{2})$, then
\[L_1L_3(P_p)=2(2\gamma+\sqrt{2}\alpha)\equiv 8\gamma \bmod p\]
and
\[\pair{\Lambda,\Lambda'}_p=\bigl[L_1L_3(P_p),d\bigr]_p=[\gamma,d]_p.\]

Note that
$n(b\beta)^2-(a\alpha)^2=2(b^2\gamma^2-c^2\alpha^2)\equiv 0\bmod{16}$, we may take $\sqrt{n}\in\BQ_2$ such that $b\beta\sqrt{n}\equiv a\alpha\bmod 8$.
Take $P_2=(1,0,c\sqrt{n},-a\sqrt{n})$, then
\[L_1L_3(P_2)
	=-c\sqrt{n}(bn\beta-2c\gamma\sqrt{n}-a\alpha\sqrt{n})
	=2c^2n\gamma+cn(a\alpha-b\beta\sqrt{n})\]
and
\[\pair{\Lambda,\Lambda'}_2=\bigl[L_1L_3(P_2),d\bigr]_2
	=[2c^2n\gamma,d]_2=[\gamma,d]_2
	=\aleg{-1}d \aleg{-1}\gamma.\]
Since $\alpha^2\equiv-n\beta^2\bmod\gamma$, we have $\leg{-1}\gamma=\leg n\gamma=\leg\gamma n$.
Hence
\[\pair{\Lambda,\Lambda'}=\sum_{p\mid n}\pair{\Lambda,\Lambda'}_p+\pair{\Lambda,\Lambda'}_2=\aleg\gamma d+\aleg{-1}d\aleg\gamma n.\]

Since $\bfR_n=(\bfA,{\bf0})$, we have $\CA[2]\cap \CA^2=\set{\bigl[(1)\bigr],\bigl[(2,\sqrt{-n})\bigr]}$.
Since $\Ker\bfA^\rmT=\set{{\bf0},{\bfd+\aleg{-1}{d}{\bf1}}}$ by Lemma~\ref{lem:pm1-mod-8}(1), we have
\[\Im\bfR_n=\Im\bfA=\set{\bfu:\bfu^\rmT\Bigl(\bfd+\aleg{-1}{d}{\bf1}\Bigr)=0}.\]
By Proposition~\ref{pro:8rank}, $[(2,\sqrt{-n})]\in \CA^4$ if and only if
\[\bfb_\gamma=\biggl(\aleg \gamma{p_1},\dots,\aleg \gamma{p_k}\biggr)^\rmT\in\Im\bfR_n,\]
if and only if
\[\pair{\Lambda,\Lambda'}
=\aleg{\gamma}{d}+\aleg{-1}{d}\aleg{\gamma}{n}
=\bfb_\gamma^\rmT\Bigl(\bfd+\aleg{-1}{d}{\bf1}\Bigr)=0.\]
In conclusion, the Cassels pairing is non-degenerate if and only if $h_8(n)=0$.
\end{proof}

\subsection{Proof of Theorem~\ref{thm:main}(B)}

\begin{lemma}\label{lem:1-mod-4}
Assume that each $p_i\equiv1\bmod 4$ and $\Sel_2(E/\BQ)\cong(\BZ/2\BZ)^2$.
Let $\bfd=(s_1,\cdots,s_k)^\rmT$ be a column vector in $\BF_2^k$ and $d=p_1^{s_1}\cdots p_k^{s_k}$.
\begin{enumerate}
\item $\dim_{\BF_2}\Sel_2'\bigl(E^{(n)}\bigr)=2$ if and only if $h_4(n)=1$.
In which case, $\rank\bfA=k-2$ or $k-1$.
\item If $h_4(n)=1$ and $\rank\bfA=k-2$, then $\Sel_2'\bigl(E^{(n)}\bigr)$ is generated by $(d,d,1)$ and $(-1,1,-1)$, where $\Ker\bfA=\set{{\bf0},{\bf1},\bfd,\bfd+{\bf1}}$.
Moreover, $d\equiv5\bmod8$.
\item If $h_4(n)=1$ and $\rank\bfA=k-1$, then $\Sel_2'\bigl(E^{(n)}\bigr)$ is generated by $(2d,2d,1)$ and $(-1,1,-1)$, where $\bfA\bfd=\bfb_2$.
\end{enumerate}
\end{lemma}

\begin{proof}
Similar to the proof of Lemma~\ref{lem:pm1-mod-8}(2), we have $\Ker\CM_1=0$.
It suffices to show that $\rank\bfM_n=2k-2$ if and only if the R\'edei matrix $\bfR_n=(\bfA,\bfb_2)$ has rank $k-1$ by Proposition~\ref{pro:4rank}.
Since $\bfA{\bf1}={\bf0}$, we have $\rank\bfA\le k-1$.
If $\rank\bfM_n=2k-2$, then
\[2k-2=\rank\begin{pmatrix}
\bfA+\bfD_2&\bfD_2\\\bfD_2&\bfA+\bfD_2
\end{pmatrix}
=\rank\begin{pmatrix}
\bfA&\bfD_2\\&\bfA\end{pmatrix}\le k+\rank\bfA\]
and $\rank\bfA\ge k-2$.
If $\rank\bfR_n=k-1$, then clearly $\rank\bfA\ge k-2$.

Suppose that $\rank\bfA=k-2$.
If $\rank\bfM_n=2k-2$, then $\bfb_2\notin\Im\bfA$.
Otherwise assume that $\bfA\bfa=\bfb_2$, then
\[\Ker\bfM_n\supseteq\set{\svec{\bfu}{\bfu},\svec{\bfu+\bfa}{\bfu+\bfa+{\bf1}}:\bfu\in\Ker\bfA}\]
has at least $8$ elements, which is impossible.
Therefore, $\rank\bfR_n=\rank(\bfA,\bfb_2)=k-1$.
Conversely, if $\rank\bfR_n=k-1$, then $\bfb_2\notin \Im\bfA$.
Since $n\equiv1\bmod 8$, we have ${\bf1}^\rmT\bfb_2=0$.
Note that $\bfA$ is symmetric, we have
\[\Im\bfA=\set{\bfu:{\bf1}^\rmT\bfu=\bfd^\rmT\bfu=0},\]
$\bfd^\rmT\bfb_2=1$ and ${\bf1}^\rmT\bfD_2(\bfd+{\bf1})={\bf1}^\rmT\bfD_2\bfd=\bfb_2^\rmT\bfd=1$.
Hence $\bfD_2{\bf1},\bfD_2\bfd,\bfD_2(\bfd+{\bf1})\notin\Im\bfA$.
If $\svec{\bfx}{\bfy}\in\Ker\bfM_n$, then $\bfx+\bfy\in\Ker\bfA$ and $\bfD_2(\bfx+\bfy)=\bfA\bfx$.
This forces $\bfx+\bfy={\bf0}$ and $\bfx=\bfy\in\Ker\bfA$.
Hence $\#\Ker\bfM_n=\#\Ker\bfA=4$ and $\rank\bfM_n=2k-2$.
In this case,
\[\Ker\CM_n=\set{
\svecc{\bf0}{\bf0}{\bf0},
\svecc{\bf1}{\bf1}{\bf0},
\svecc{\bfd}{\bfd}{\bf0},
\svecc{\bfd+\bf1}{\bfd+\bf1}{\bf0}}.\]
In other words, $\Sel_2'\bigl(E^{(n)}\bigr)$ is generated by $(n,n,1)$ and $(d,d,1)$.
Since $\bfd^\rmT\bfb_2=1$, we have $\leg{2}{d}=1$ and $d\equiv 5\bmod8$.

Suppose that $\rank\bfA=k-1$.
Then $\Ker\bfA=\set{{\bf0},{\bf1}}$ and $\Im\bfA=\set{\bfu:{\bf1}^\rmT\bfu=0}$.
Since $n\equiv1\bmod 8$, we have ${\bf1}^\rmT\bfb_2=0$ and $\bfb_2\in\Im\bfA$.
Thus $\rank\bfR_n=k-1$, $h_4(n)=1$ and
\[\Ker\CM_n=\set{
\svecc{\bf0}{\bf0}{\bf0},
\svecc{\bf1}{\bf1}{\bf0},
\svecc{\bfd}{\bfd+\bf1}{\bf0},
\svecc{\bfd+\bf1}{\bfd}{\bf0}}.\]
In this case, $\Sel_2'\bigl(E^{(n)}\bigr)$ is generated by $(n,n,1)$ and $(d,nd,n)$.

Conclude the result by the fact that $(n,n,1)-(-1,1,-1)=(-n,n,-1)$ and $(d,nd,n)-(2d,2d,1)=(2,2n,n)$ correspond torsions, see~\eqref{eq:torsion-homogeneous}.
\end{proof}

\begin{theorem}\label{thm:main_B}
Assume that $\Sel_2(E/\BQ)\cong(\BZ/2\BZ)^2$.
Let $n$ be a positive square-free integer prime to $abc$ where each prime factor of $n$ is a quadratic residue modulo every prime factor of $abc$.
If all prime factors of $n\equiv1\bmod8$ are congruent to $1$ modulo $4$, then the following are equivalent:
\begin{enumerate}
\item $\rank_\BZ E^{(n)}(\BQ)=0$ and $\Sha(E^{(n)}/\BQ)[2^\infty]\cong(\BZ/2\BZ)^2$;
\item $h_4(n)=1$ and $h_8(n)\equiv\frac{d-1}4\bmod2$.
\end{enumerate}
Here $d$ is the odd part of $d_0\mid 2n$ such that the ideal class $[(d_0,\sqrt{-n})]$ is the non-trivial element in $\CA[2]\cap\CA^2$.
\end{theorem}

\begin{proof}
By Lemma~\ref{lem:non-deg}, (1) is equivalent to say, $\Sel_2'\bigl(E^{(n)}\bigr)$ has dimension $2$ and the Cassels pairing on it is non-degenerate.
By Lemma~\ref{lem:1-mod-4}(1), $\dim_{\BF_2}\Sel_2'\bigl(E^{(n)}\bigr)=2$ if and only if $h_4(n)=1$.
Assume that $h_4(n)=1$.

(1) The case $\rank\bfA=k-2$.
By Lemma~\ref{lem:1-mod-4}(2) and Proposition~\ref{pro:4rank}, we have $\bfb_2\notin\Im\bfA$ and $\CD(K)\cap\bfN_{K/\BQ}K^\times=\set{1,n,d,n/d}$ with $d=d_0\equiv 5\bmod8$.
Denote by $d'=n/d\equiv 5\bmod8$.
Since $d$ is a norm, there exists a primitive triple $(\alpha,\beta,\gamma)$ of positive integers such that
\[d\alpha^2+d'\beta^2=\gamma^2.\]
If $\alpha$ is odd, then $\beta$ is even and the triple
\[(\alpha',\beta',\gamma')=\biggl(\Bigl|\frac{(d-d')\alpha}2+d'\beta\Bigr|,\Bigl|\frac{(d-d')\beta}2-d\alpha\Bigr|,\frac{(d+d')\gamma}{2}\biggr)\]
is another primitive solution with even $\alpha'$.
Thus we may assume that $\alpha$ is even.
Then all of $\alpha/2,\beta,\gamma$ are odd since $d'\equiv 5\bmod8$.

By Lemma~\ref{lem:1-mod-4}(2), $\Sel_2'\bigl(E^{(n)}\bigr)$ is generated by $\Lambda=(d,d,1)$ and $\Lambda'=(-1,1,-1)$.
Recall that $D_\Lambda$ is
\[\begin{cases}
	H_1:& -b^2nt^2+du_2^2-u_3^2=0,\\
	H_2:& -a^2nt^2+u_3^2-du_1^2=0,\\
	H_3:&2c^2d't^2+u_1^2-u_2^2=0.
\end{cases}\]
Choose
\[\begin{aligned}
Q_1&=(\beta,b\gamma,bd\alpha)\in H_1(\BQ),&
L_1&=bd'\beta t-\gamma u_2+\alpha u_3,\\
Q_3&=(0,1,1)\in H_3(\BQ),&
L_3&=u_1-u_2.
\end{aligned}\]
By Lemma~\ref{lem:cassels}, we have
\[\pair{\Lambda,\Lambda'}=\sum_{p\mid 2nabc\infty} \bigl[L_1L_3(P_p),-1\bigr]_p\]
for any $P_p\in D_\Lambda(\BQ_p)$.
For each $p\mid n$, we have $p\equiv1\bmod4$ and then $\pair{\Lambda,\Lambda'}_p=0$.
Since for any $q\mid c$, we have $-a^2=b^2-2c^2\equiv b^2\bmod q$, we have $q\equiv1\bmod4$ and then $\pair{\Lambda,\Lambda'}_q=0$.

Take $P_\infty=(t,u_1,u_2,u_3)=(0,1,-1,\sqrt{d})$, then
\[L_1L_3(P_\infty)=2(\gamma+\alpha\sqrt d)>0\]
and
\[\pair{\Lambda,\Lambda'}_\infty=\bigl[L_1L_3(P_\infty),-1\bigr]_\infty=0.\]

Take $P_2=(2,\sqrt{1-8c^2d'},1,\sqrt{d-4b^2n})$ where $u_1\equiv3\bmod8$.
Note that $bd'\beta+\alpha u_3/2$ is even.
We have
\[L_1L_3(P_2)=(u_1-1)(2bd'\beta+\alpha u_3-\gamma)\]
and
\[\pair{\Lambda,\Lambda'}_2=\bigl[L_1L_3(P_2),-1\bigr]_2=[2,-1]_2+[-\gamma,-1]_2=\aleg{-1}\gamma+1.\]
Since $d\alpha^2\equiv-d'\beta^2\bmod \gamma$, we have $\leg{-1}\gamma=\leg n\gamma=\leg\gamma n$ and $\pair{\Lambda,\Lambda'}_2=\aleg\gamma n+1$.

For $q\mid ab$, take $P_q=(0,1,-1,\sqrt{d})$.
Since $\gamma^2-d\alpha^2=d'\beta^2$, we may choose $\sqrt{d}$ such that $q\mid(\gamma-\alpha\sqrt{d})$ if $q\mid \beta$.
Then
\[L_1L_3(P_q)=2(\gamma+\alpha\sqrt{d})\in\BZ_q^\times\]
and
\[\pair{\Lambda,\Lambda'}_q=\bigl[L_1L_3(P_q),-1\bigr]_q=0.\]
Hence
\[\pair{\Lambda,\Lambda'}=\pair{\Lambda,\Lambda'}_2=\aleg\gamma n+1.\]

Since $\bfR_n=(\bfA,\bfb_2)$, we have $\CA[2]\cap\CA^2=\set{[(1)],[(d,\sqrt{-n})]}$.
Since $\bfb_2\notin\Im\bfA$ and $\bfA{\bf1}={\bf0}$, we have
\[\Im\bfR_n=\set{\bfu:{\bf1}^\rmT\bfu=0}.\]
By Lemma~\ref{pro:8rank}, $[(d,\sqrt{-n})]\in \CA^4$ if and only if
\[\bfb_\gamma=\Bigl(\aleg\gamma{p_1},\dots,\aleg\gamma{p_k}\Bigr)^\rmT\in\Im\bfR_n,\]
if and only if
\[\pair{\Lambda,\Lambda'}=\aleg\gamma n+1={\bf1}^\rmT\bfb_\gamma+1=1.\]
In conclusion, the Cassels pairing is non-degenerate if and only if $h_8(n)=1=\aleg{2}{d}$.

(2) The case $\rank\bfA=k-1$.
By Lemma~\ref{lem:1-mod-4}(3) and Proposition~\ref{pro:4rank}, we have $\bfb_2\in\Im\bfA$ and $\CD(K)\cap\bfN_{K/\BQ}K^\times=\set{1,n,2d,2n/d}$.
Denote by $d'=n/d$.
Since $d_0=2d$ is a norm, there exists a primitive triple $(\alpha,\beta,\gamma)$ of positive integers such that
\[d\alpha^2+d'\beta^2=2\gamma^2.\]
It's easy to see that all of $\alpha,\beta,\gamma$ are odd.

By Lemma~\ref{lem:1-mod-4}(3), $\Sel_2'\bigl(E^{(n)}\bigr)$ is generated by
$\Lambda=(2d,2d,1)$ and $\Lambda'=(-1,1,-1)$.
Recall that $D_\Lambda$ is
\[\begin{cases}
	H_1:& -b^2nt^2+2du_2^2-u_3^2=0,\\
	H_2:& -a^2nt^2+u_3^2-2du_1^2=0,\\
	H_3:&c^2d't^2+u_1^2-u_2^2=0.
\end{cases}\]
Choose
\[\begin{aligned}
Q_1&=(\beta,b\gamma,bd\alpha)\in H_1(\BQ),&
L_1&=bd'\beta t-2\gamma u_2+\alpha u_3,\\
Q_3&=(0,1,1)\in H_3(\BQ),&
L_3&=u_1-u_2.
\end{aligned}\]
Similar to the case $\rank\bfA=k-2$, we have
\[\pair{\Lambda,\Lambda'}=\sum_{p\mid 2ab\infty} \bigl[L_1L_3(P_p),-1\bigr]_p\]
for any $P_p\in D_\Lambda(\BQ_p)$.

For $p=\infty$, take $P_\infty=(0,1,-1,\sqrt{2d})$.
Then
\[L_1L_3(P_\infty)=2(2\gamma+\alpha \sqrt{2d})>0\]
and
\[\pair{\Lambda,\Lambda'}_\infty=\bigl[L_1L_3(P_\infty),-1\bigr]_\infty=0.\]

For $p=2$, take $P_2=(t,u_1,u_2,u_3)$ where
\[t=1,\ u_1=2\aleg2d,\ u_2^2=c^2d'+u_1^2,\ u_3^2=a^2n+2du_1^2\]
with $\gamma u_2\equiv 1\bmod 4$.
Since
\[\begin{split}
&(bd'\beta+\alpha u_3)(bd'\beta-\alpha u_3)
=b^2{d'}^2\beta^2-\alpha^2(a^2n+2du_1^2)\\
=&b^2d'(2\gamma^2-d\alpha^2)-\alpha^2(a^2n+2du_1^2)
=2b^2d'\gamma^2-\alpha^2(2c^2n+2du_1^2)\\
=&2\bigl((bd'\gamma)^2-n\alpha^2u_2^2\bigr)/d'\equiv0\bmod{16},
\end{split}\]
we may choose $u_3$ such that $8\mid bd'\beta+\alpha u_3$.
Then
\[\begin{split}
\pair{\Lambda,\Lambda'}_2=&\bigl[L_1L_3(P_2),-1\bigr]_2
=\bigl[(u_1-u_2)(bd'\beta+\alpha u_3-2\gamma u_2),-1\bigr]_2\\
=&\bigl[-2\gamma u_2(u_1-u_2),-1\bigr]_2
=[2,-1]_2+\bigl[u_2-u_1,-1\bigr]_2\\
=&[\gamma,-1]_2+[1-u_1\gamma,-1]_2\\
=&\aleg{-1}\gamma+\Bigl[1-2\aleg2d,-1\Bigr]_2=\aleg{-1}\gamma+\aleg2d.
\end{split}\]
Since $d\alpha^2\equiv-d'\beta^2\bmod \gamma$, we have $\leg{-1}\gamma=\leg n\gamma=\leg\gamma n$ and $\pair{\Lambda,\Lambda'}_2=\aleg\gamma n+\aleg2d$.

For $q\mid a$, take $P_q=(1,0,u_2,a\sqrt n)$ where $u_2^2=c^2{d'}$.
Since
\[\begin{split}
&(bd'\beta-2\gamma u_2)(bd'\beta+2\gamma u_2)
=b^2{d'}^2\beta^2-4c^2d'\gamma^2\\
\equiv&2c^2d'(d'\beta^2-2\gamma^2)=-2c^2n\alpha^2\bmod q,
\end{split}\]
we may choose $u_2$ such that $q\mid bd'\beta+2\gamma u_2$ if $q\mid\alpha$.
If $q\mid bd'\beta\pm 2\gamma u_2$, then $q\mid \beta$, which contradicts to the primitivity of $(\alpha,\beta,\gamma)$.
Therefore, $q\nmid bd'\beta-2\gamma u_2$.
If $q\nmid\alpha$, clearly we have $q\nmid bd'\beta\pm2\gamma u_2$.
Then
\[L_1L_3(P_q)=-u_2(bd'\beta-2\gamma u_2+a\alpha\sqrt n)\in\BZ_q^\times\]
and
\[\pair{\Lambda,\Lambda'}_q=\bigl[L_1L_3(P_q),-1\bigr]_q=0.\]
Similarly, $\pair{\Lambda,\Lambda'}_q=0$ for $q\mid b$.
Hence
\[\pair{\Lambda,\Lambda'}=\pair{\Lambda,\Lambda'}_2=\aleg\gamma n+\aleg2d.\]

Since $\bfR_n=(\bfA,\bfb_2)$, we have $\CA[2]\cap\CA^2=\set{[(1)],[(2d,\sqrt{-n})]}$.
Since $\bfb_2\in\Im\bfA$, we have
\[\Im\bfR_n=\Im\bfA=\set{\bfu:{\bf1}^\rmT\bfu=0}.\]
By Lemma~\ref{pro:8rank}, $[(2d,\sqrt{-n})]\in \CA^4$ if and only if
\[\bfb_\gamma=\Bigl(\aleg\gamma{p_1},\dots,\aleg\gamma{p_k}\Bigr)^\rmT\in\Im\bfR_n,\]
if and only if
\[\pair{\Lambda,\Lambda'}=\aleg\gamma n+\aleg{2}{d}={\bf1}^\rmT\bfb_\gamma+\aleg{2}{d}=\aleg{2}{d}.\]
In conclusion, the Cassels pairing is non-degenerate if and only if $h_8(n)=\aleg{2}{d}$.
\end{proof}

\section{Equidistribution of residue symbols}

In this and next sections, we will prove Theorem~\ref{thm:dist}.

\subsection{Residue symbols}
\label{Residue symbols}

\begin{definition}
Denote by $I=\sqrt{-1}$ and $\BZ[I]$ the ring of Gauss integers.

(1) A prime element $\lambda$ of $\BZ[I]$ is called {\em Gaussian} if it is not a rational prime.

(2) An integer $\lambda\in\BZ[I]$ is called {\em primary} if $\lambda\equiv1\bmod{(2+2I)}$.
\end{definition}

Recall the quadratic and quartic residue symbols on $\BZ[I]$, see \cite[p.~196]{Hecke1981} and \cite{IrelandRosen1990}.
Denote by $\bfN=\bfN_{\BQ(I)/\BQ}$ the norm from $\BQ(I)$ to $\BQ$.
For any $\alpha\in\BZ[I]$ and prime element $\lambda$ prime to $1+I$, define
\begin{equation}\label{eq:quadratic-residue-symbol}
\leg\alpha\lambda_2\in\set{0,\pm1}\quad\text{such that}\quad \leg\alpha\lambda_2\equiv\alpha^{\frac{\bfN\lambda-1}2}\bmod\lambda.
\end{equation}
For any element $\lambda$ prime to $1+I$ with a prime decomposition $\lambda=\prod_{i=1}^k\lambda_k$, define $\leg\alpha\lambda_2=\prod_{i=1}^k\leg\alpha{\lambda_i}_2$.

For any $\alpha\in\BZ[I]$ and primary prime $\lambda$, define
\begin{equation}\label{eq:quartic-residue-symbol}
\leg\alpha\lambda_4\in\set{0,\pm1,\pm I}\quad\text{such that}\quad \leg\alpha\lambda_4\equiv\alpha^{\frac{\bfN\lambda-1}4}\bmod\lambda.
\end{equation}
For any primary element $\lambda$ with a primary prime decomposition $\lambda=\prod_{i=1}^k\lambda_k$, define $\leg\alpha\lambda_4=\prod_{i=1}^k\leg\alpha{\lambda_i}_4$.
Let $\lambda$ and $\lambda'$ be two coprime primary primes.
Then we have the quartic reciprocity law
\[\leg\lambda{\lambda'}_4=\leg{\lambda'}\lambda_4 (-1)^{\frac{\bfN\lambda-1}4\cdot\frac{\bfN\lambda'-1}4}.\]
Certainly, $\leg\alpha\lambda_2=\leg\alpha\lambda_4^2$.

Let $p\equiv1\bmod4$ be a rational prime.
Let $a$ be a rational integer such that $\leg ap=1$.
By abuse of notations, we define
\begin{equation}\label{eq:rational-quartic-residue-symbol}
\leg ap_4:=\leg a\lambda_4,
\end{equation}
where $\lambda$ is a primary prime such that $\bfN\lambda=p$.
For any rational integer $d= p_1\cdots p_k$ with $p_i\equiv1\bmod4$, define $\leg ad_4=\prod_{i=1}^k\leg a{p_i}_4$.

\subsection{Analytic results}
\label{Analytic results}

Let $F$ be a number field with degree $n$, discriminant $\Delta$ and ring of integers $\CO$.
Denote by $\bfN=\bfN_{F/\BQ}$ the norm from $F$ to $\BQ$.

For an ideal $\ff$ of $\CO$, denote by $I(\ff)$ the group of fractional ideals prime to $\ff$ and $P_\ff$ the subgroup consisting of principal fractional ideals $(\gamma)=\gamma\CO$ with totally real $\gamma\equiv1\bmod\ff$.
A character $\chi$ of $I(\ff)/P_\ff$ is called a {\em character modulo $\ff$}.
It can be viewed as a character on $I(\ff)$.
If $\fa$ is a fractional ideal not coprime to $\ff$, define $\chi(\fa)=0$.
Denote by
\begin{equation}\label{eq:Mangoldt-function}
\Lambda(\fa)=\begin{cases}
\log\bfN\fp& \text{if $\fa=\fp^m$ with $m\ge 1$};\\
0 & \text{otherwise}
\end{cases}
\end{equation}
the {\em Mangoldt function}.
Define
\begin{equation}\label{eq:psi-x-chi}
\psi(x,\chi)=\sum_{\bfN\fa\le x}\chi(\fa)\Lambda(\fa).
\end{equation}
Denote by $\chi_0$ the principal character on $I(\ff)/P_\ff$.

\begin{proposition}[{\cite[p.~112, Exercise~7]{IwaniecKowalski2004}}]
\label{pro:psi-x-chi-formula}
If $\chi\neq\chi_0$ is a character modulo $\ff$ and $1\le T\le x$, then
\[\psi(x,\chi)=-\sum_{|\Im\rho|\le T}\frac{x^\rho-1}\rho+
O\bigl(T^{-1}x\log x \log(x^n\bfN\ff)\bigr).\]
Here $\rho$ runs over all the zeros of $L(s,\chi)$ with $0\le \Re\rho\le1$.
\end{proposition}

Similar to the classical process on the estimation of $\psi(x,\chi)$ as in \cite[\S~19]{Davenport1980}, we derive the following explicit formula
\begin{equation}\label{eq:explicit-formula-of-psi-x-chi}
\psi(x,\chi)=-\frac{x^{\beta'}}{\beta'}+ R(x,T)
\end{equation}
with
\[R(x,T)\ll x\log^2(x\bfN\ff)\exp\biggl(-\frac{c_1\log x}{\log(T\bfN\ff)}\biggr)+T^{-1}x\log x\cdot\log(x^n\bfN\ff)+x^{\frac14}\log x.\]
We also use the estimation on the number of zeroes in \cite[Satz~LXXI]{Landau1918}.
Here $c_1$ is a positive constant and the term $-\frac{x^{\beta'}}{\beta'}$ occurs only if $\chi$ is a real character such that $L(s,\chi)$ has a zero $\beta'$  satisfying
\[\beta'>1-\frac {c_2}{\log\bfN\ff}\]
with $c_2$  a positive constant.

The Siegel Theorem over $F$ as follows is \cite[Theorem]{Fogels1961} and \cite[Satz]{Fogels1963}.
\begin{proposition}\label{pro:Siegel-theorem}
Let $\chi$ be a character modulo an integral $\ff$ and $D=|\Delta|\bfN\ff>1$.
\begin{enumerate}
\item There is a positive constant $c_3=c_3(n)$ such that in the region
\[\Re(s)> 1-\frac{c_3}{\log D(2+|\Im s|)}>\frac34\]
there is no zero of $L(s,\chi)$ in the case of a complex $\chi$. For at most one real $\chi'$, there may be a simple zero $\beta'$ of $L(s,\chi')$ in this region.
\item For any $\varepsilon>0$, there exists a positive constant $c_4=c_4(n,\varepsilon)$ such that
\[1-\beta'> c_4(n,\varepsilon) D^{-\varepsilon}.\]
\end{enumerate}
\end{proposition}

The Page Theorem over $F$ as follows is a special case of \cite[\S~3, Theorem~A]{HoffsteinRamakrishnan1995}.
\begin{proposition}\label{pro:Page-theorem}
For any $Z\ge 2$ and a suitable constant $c_5$, there is at most a real primitive character $\chi$ modulo $\ff$ with $\bfN\ff\le Z$ such that $L(s,\chi)$ has a real zero $\beta$ satisfying
\[\beta >1-\frac{c_5}{\log Z}.\]
\end{proposition}

\subsection{Equidistribution of residue symbols}
\label{Equidistribution of residue symbols}

Recall that $abc=q_1^{t_1}\cdots q_\ell^{t_\ell}$ is the prime decomposition of $abc$.
Let $\alpha=(\alpha_1,\cdots,\alpha_k)$ be a vector with $\alpha_i\in\set{1,5,9,13}$ and $\alpha_1\cdots\alpha_k\equiv1\bmod8$.
Let $\bfB=(B_{ij})_{k\times k}\in M_k(\BF_2)$ be a symmetric matrix with rank $k-2$ and $\bfB{\bf1}={\bf0}$.
Then $\Ker\bfB=\set{{\bf0},{\bf1},\bfd,\bfd+{\bf1}}$ for some vector $\bfd=(s_1,\cdots,s_k)^\rmT$ with $s_k=0$.

Denote by $C_k(x,\alpha,\bfB)$ the set of all $n=p_1\cdots p_k$ satisfying
\begin{itemize}
\item $n\le x$ and $p_1<\cdots<p_k$;
\item $p_i\equiv\alpha_i\bmod{16}$ for all $1\le i\le k$;
\item $\aleg{p_j}{p_i}=B_{ij}$ for all $1\le i<j\le k$;
\item $\leg{p_i}{q_j}=1$ for all $1\le i\le k$ and $1\le j\le\ell$;
\item $\leg{d'}d_4\leg d{d'}_4=-1$, where $d=p_1^{s_1}\cdots p_k^{s_k}$ and $d'=n/d$,
\end{itemize}
and denote by $C_k'(x,\alpha,\bfB)$ the set of all $\eta=\lambda_1\cdots\lambda_k$ satisfying
\begin{itemize}
\item $\bfN\eta\le x$ and $\bfN\lambda_1<\cdots<\bfN\lambda_k$;
\item $\lambda_i\in\CP$ and $\bfN\lambda_i\equiv\alpha_i\bmod{16}$ for all $1\le i\le k$;
\item $\aleg{\bfN\lambda_j}{\bfN\lambda_i}=B_{ij}$ for all $1\le i<j\le k$;
\item $\leg{\bfN\lambda_i}{q_j}=1$ for all $1\le i\le k$ and $1\le j\le\ell$;
\item $\leg{\delta'}{\delta}_2=-1$, where $\delta=\lambda_1^{s_1}\cdots\lambda_k^{s_k}$ and $\delta'=\eta/\delta$.
\end{itemize}
Here, $\CP$ is the set of primary primes in $\BZ[I]$ with positive imaginary part.

In this section, we will give an estimation of the number of $C_k(x,\alpha,\bfB)$.

\begin{lemma}\label{lem:identify-Ck}
There is a bijection
\[C_k'(x,\alpha,\bfB)\longrightarrow C_k(x,\alpha,\bfB),\quad \eta\mapsto\bfN\eta.\]
\end{lemma}
\begin{proof}
For any $\eta=\lambda_1\cdots\lambda_k\in C_k'(x,\alpha,\bfB)$,  denote by $p_i=\bfN\lambda_i$.
By the quartic reciprocity law, we have
\[\begin{split}
&\leg{p_i}{p_j}_4\leg{p_j}{p_i}_4
 =\leg{\lambda_i\ov{\lambda_i}}{\lambda_j}_4
 \leg{\lambda_j\ov{\lambda_j}}{\lambda_i}_4
 =\leg{\lambda_i}{\lambda_j}_4\leg{\ov{\lambda_i}}{\lambda_j}_4
 \leg{\lambda_j}{\lambda_i}_4\leg{\ov{\lambda_j}}{\lambda_i}_4\\
=&\leg{\lambda_j}{\lambda_i}_4\leg{{\lambda_j}}{\ov{\lambda_i}}_4
 \leg{\lambda_j}{\lambda_i}_4\leg{\ov{\lambda_j}}{\lambda_i}_4
 =\leg{\lambda_j}{\lambda_i}_2
 \ov{\leg{{\ov{\lambda_j}}}{\lambda_i}_4}\leg{\ov{\lambda_j}}{\lambda_i}_4
 =\leg{\lambda_j}{\lambda_i}_2.
\end{split}\]
Therefore,
\[\leg{d'}d_4\leg d{d'}_4=\leg{\delta'}{\delta}_2=-1,\]
where $d=\bfN\delta$ and $d'=\bfN\delta'$.
Hence $\bfN\eta\in C_k(x,\alpha,\bfB)$.

For any rational prime $p\equiv1\bmod4$, there is exactly one primary prime in $\CP$ with norm $p$.
This gives the surjectivity.
The injectivity is trivial.
\end{proof}

Denote by $T_k(x)$ the set of all $n=p_1\cdots p_{k-1}$ satisfying
\begin{itemize}
\item $n\le x$ and $p_1<\cdots <p_{k-1}$;
\item $p_i\equiv\alpha_i\bmod{16}$ for all $1\le i\le k-1$;
\item $\aleg{p_j}{p_i}=B_{ij}$ for all $1\le i<j\le k-1$;
\item $\leg{p_i}{q_j}=1$ for all $1\le i\le k-1$ and $1\le j\le \ell$,
\end{itemize}
and denote by $T'_k(x)$ the set of all $\eta=\lambda_1\cdots\lambda_{k-1}$ satisfying
\begin{itemize}
\item $\bfN\eta\le x$ and $\bfN\lambda_1<\cdots<\bfN\lambda_{k-1}$;
\item $\lambda_i\in\CP$ and $\bfN\lambda_i\equiv\alpha_i\bmod{16}$ for all $1\le i\le k-1$;
\item $\aleg{\bfN\lambda_j}{\bfN\lambda_i}=B_{ij}$ for all $1\le i<j\le k-1$;
\item $\leg{\bfN\lambda_i}{q_j}=1$ for all $1\le i<k$ and $1\le j\le\ell$.
\end{itemize}
The independence property of Legendre symbols in \cite{Rhoades2009} implies that
\begin{equation}\label{eq:number-Tk(x)}
\#T_k(x)\sim 2^{-(\ell+3)(k-1)-\binom{k-1}2}\cdot\#C_{k-1}(x),
\end{equation}
where $C_k(x)$ is the set of all positive square-free integers $n\le x$ with exactly $k$ prime factors.

\begin{lemma}\label{lem:identify-Tk(x)}
There is a bijection
\[T'_k(x)\longrightarrow T_k(x),\quad \eta\mapsto\bfN\eta.\]
\end{lemma}
\begin{proof}
For any rational prime $p\equiv1\bmod4$, there is exactly one primary prime in $\CP$ with norm $p$.
This proves the surjectivity.
The injectivity is trivial.
\end{proof}

\begin{theorem}\label{thm:independence}
Notations as above with $k>1$.
We have
\[\#C_k(x,\alpha,\bfB)\sim 2^{-k\ell-3k-1-\binom k2}\cdot\#C_k(x),\]
where $C_k(x)$ is the set of all positive square-free integers $n\le x$ with exactly $k$ prime factors.
\end{theorem}
\begin{proof}
Similar to \cite{CremonaOdoni1989}, we consider the comparison map
\[f: C_k'(x,\alpha,\bfB)\longrightarrow T'_k(x),\quad \lambda_1\cdots\lambda_k\mapsto\lambda_1\cdots\lambda_{k-1}.\]
Let $Q_1$ be the product of all primary primes $\mu\in\CP$ dividing $abc$, and $Q_2$ the product of all prime $q\mid abc$ with $q\equiv3\bmod4$.
For any $\eta=\lambda_1\cdots\lambda_{k-1}\in T_k'(x)$, denote by $\fc_\eta=16\bfN(\eta Q_1)Q_2\BZ[I]$.
It's easy to see that if $\beta$ satisfies
\begin{itemize}
\item $\bfN\beta\equiv\alpha_k\bmod{16}$;
\item $\aleg{\bfN\beta}{\bfN\lambda_i}=B_{ik}$ for all $1\le i\le k-1$;
\item $\leg{\bfN\beta}{q_j}=1$ for all $1\le j\le\ell$;
\item $\leg{\beta}\delta_2=-\leg{\eta/\delta}\delta_2$, where $\delta=\lambda_1^{s_1}\cdots\lambda_k^{s_k}$,
\end{itemize}
then so is $\beta'\equiv \beta\bmod 16\bfN(\eta Q_1)Q_2$.
Denote by
\[\msa_\eta\subseteq(\BZ[I]/\fc_\eta)^\times\]
the classes of such $\beta$.
Then $\eta$ lies in the image of $f$ if and only if there exists $\theta\in\CP$ such that $\bfN\lambda_{k-1}<\bfN\theta\le x/\bfN\eta$ and $\theta\mod\fc_\eta\in\msa_\eta$ by noting that $s_k=0$.

\begin{lemma}\label{lem:size-of-fiber}
Let $\chi_1,\chi_2:G\to\BF_2$ be two different non-trivial quadratic character on a finite group $G$.
Then the size of $\chi_1^{-1}(i)\cap\chi_2^{-1}(j)$ is $\#G/4$ for any $i,j\in\BF_2$.
\end{lemma}
\begin{proof}
The sizes of $\chi_1^{-1}(i)$ and $\chi_2^{-1}(j)$ are $\#G/2$.
Since $\chi_1\neq\chi_2$, these two sets always have a common element, which means that $(\chi_1,\chi_2):G\to\BF_2^2$ is surjective.
The result then follows.
\end{proof}

\begin{lemma}\label{lem:different-characters}
Assume that $\pi\in\CP$ and $p=\bfN\pi$.
Then $\leg{x}{\pi}_2$ and $\leg{\bfN x}{p}$ are different non-trivial quadratic characters on $\bigl(\BZ[I]/p\BZ[I]\bigr)^\times$.
\end{lemma}
\begin{proof}
Since $\bfN:\bigl(\BZ[I]/p\BZ[I]\bigr)^\times\to(\BZ/p\BZ)^\times$ is surjective, $\leg{\bfN x}{p}$ is non-trivial.
Let $\gamma\in\BZ[I]$ be an element such that $\pi\gamma\equiv1\bmod\ov\pi$.
Let $x=\ov{\pi\gamma}+\alpha\pi\gamma$ for some $\alpha\in\BZ$ coprime to $p$.
Then
\[\leg x\pi_2=\leg{\ov{\pi\gamma}}\pi_2=1.\]
Denote by $A=(\pi\gamma)^2+(\ov{\pi\gamma})^2$.
Then $\bfN(x)\equiv \alpha A\bmod p$ and
\[\leg{\bfN x}{p}=\leg{\alpha A}{p}.\]
Hence $\leg{x}{\pi}_2\neq\leg{\bfN x}{p}$ by taking $\leg\alpha p=-\leg Ap$.
\end{proof}

\begin{lemma}\label{lem:number-A-eta}
Let $\varphi(\eta)$ be the cardinality of $G=(\BZ[I]/\fc_\eta)^\times$.
Then
\[\#\msa_\eta=2^{-k-\ell-4}\varphi(\eta).\]
\end{lemma}
\begin{proof}
By the Chinese Remainder Theorem, we have a natural isomorphism
\[\begin{split}
G&\cong\biggl(\frac{\BZ[I]}{16\BZ[I]}\biggr)^\times\times
\prod_{i=1}^{k-1}\biggl(\frac{\BZ[I]}{\bfN\lambda_i\BZ[I]}\biggr)^\times\times
\prod_{\mu\mid Q_1}\biggl(\frac{\BZ[I]}{\bfN\mu\BZ[I]}\biggr)^\times\times
\prod_{q\mid Q_2}\biggl(\frac{\BZ[I]}{q\BZ[I]}\biggr)^\times\\
\beta&\mapsto (\beta_0,\beta_1,\cdots,\beta_{k-1},\beta_\mu',\beta_q').
\end{split}\]
Then $\beta\in\msa_\eta$ if and only if
\begin{enumerate}
\item $\beta_0\equiv 1\bmod{2+2I}$ and $\bfN\beta_0\equiv\alpha_k\bmod{16}$;
\item $\aleg{\bfN\beta_i}{\bfN\lambda_i}=B_{ik}$ for all $1\le i\le k-1$;
\item $\leg{\bfN\beta_\mu'}{\bfN\mu}=1$ for all $\mu\mid Q_1$;
\item $\leg{\bfN\beta_q'}{q}=1$ for all $q\mid Q_2$;
\item $\prod_{s_i=1}\leg{\beta_i}{\lambda_i}_2=-\leg{\eta/\delta}\delta_2$.
\end{enumerate}
(1) selects $\frac14\times \frac14$ number of elements in $\bigl(\BZ[I]/16\BZ[I]\bigr)^\times$.
Note that $\bigl(\BZ[I]/\lambda_i\BZ[I]\bigr)^\times\cong(\BZ/\bfN\lambda_i\BZ)^\times$, each conditions in (2)--(4) selects half number of elements in each corresponding component.

To treat (5), we choose $\beta_1,\cdots,\beta_{k-1}$ as following.
Since $s_k=0$, there is some $s_j=1$ for $1\le j\le k-1$. For $i=1,2,\cdots,j-1,j+1,\cdots,k-1$, we choose $\beta_i\in \bigl(\BZ[I]/N\lambda_i\BZ[I]\bigr)^\times$ satisfying (2), and   there are half number of $\bigl(\BZ[I]/N\lambda_i\BZ[I]\bigr)^\times$ choices. With above chosen $\beta_1,\cdots,\beta_{j-1},\beta_{j+1},\cdots,\beta_{k-1}$, applying Lemmas~\ref{lem:size-of-fiber} and \ref{lem:different-characters} to $\pi=\lambda_j$, (5) and $\aleg{\bfN\beta_j}{\bfN\lambda_j}=B_{jk}$ selects $\frac14$ number of elements in $\bigl(\BZ[I]/N\lambda_j\BZ[I]\bigr)^\times$.
Hence
\[\frac{\#\msa_\eta}{\varphi(\eta)}
=\frac1{16}\times\frac1{2^{k-1}}\times\frac1{2^\ell}\times\frac12=2^{-k-\ell-4}.\qedhere\]
\end{proof}

For any $\eta\in T'_k(x)$, denote by $h(\eta)$ the number of primes $\theta\in\CP$ such that $\bfN\lambda_{k-1}<\bfN\theta\le x/\bfN\eta$ and $\theta\mod\fc_\eta\in\msa_\eta$.
Then we have
\begin{equation}\label{eq:number-C_k'}
\#C_k'(x,\alpha,\bfB)=\sum_{\eta\in T'_k(x)} h(\eta).
\end{equation}
%
Denote by
\[M_1=(\log x)^{100}
\quad\text{and}\quad
M_2=\exp\biggl(\frac{\log x}{(\log\log x)^{100}}\biggr).\]
We will use
\[\sum^*_{\bfN\eta\in S}\]
to denote a summation over $\eta\in T'_k(x)$ with $\bfN\eta\in S$.

\begin{lemma}\label{lem:Cremona}
We have
\[\begin{split}
\sum^*_{20<\bfN\eta\le M_1}\Li(x/\bfN\eta)
	&=o\biggl(\frac{x(\log\log x)^{k-1}}{\log x}\biggr),\\
\sum^*_{M_2<\bfN\eta\le x^{\frac{k-1}k}}\Li(x/\bfN\eta)
	&=o\biggl(\frac{x(\log\log x)^{k-1}}{\log x}\biggr),\\
\sum^*_{M_1<\bfN\eta\le M_2}\Li(x/\bfN\eta)
	&\sim\frac{\#T'_k(x)}{k-1}\log\log x.
\end{split}\]
\end{lemma}
\begin{proof}
The proof is similar to \cite[Lemma~3.1]{CremonaOdoni1989}.
\end{proof}

Denote by $\pi(x)$ the number of prime ideals in $\BZ[I]$ with norm less than or equal $x$.
Then the prime ideal theorem over $\BZ[I]$ tells $\pi(x)\sim \Li(x)$.
Certainly, $h(\eta)\le\pi(x/\bfN\eta)$.
Then we have
\begin{equation}\label{eq:part-h-eta}
\begin{split}
\sum^*_{\bfN\eta\le20}h(\eta)
&\ll\Li(x),\\
\sum^*_{20<\bfN\eta\le M_1}h(\eta)
&=o\biggl(\frac{x(\log\log x)^{k-1}}{\log x}\biggr),\\
\sum^*_{M_2<\bfN\eta\le x^{\frac{k-1}k}} h(\eta)
&=o\biggl(\frac{x(\log\log x)^{k-1}}{\log x}\biggr)
\end{split}
\end{equation}
by Lemma~\ref{lem:Cremona}.
If $\bfN\eta>x^{\frac{k-1}k}$, then $\bfN\lambda_{k-1}>x^{\frac1k}$ and $x/\bfN\eta<x^{\frac1k}<\bfN\lambda_{k-1}$.
Therefore, $h(\eta)=0$ and
\begin{equation}\label{eq:part2-h-eta}
\sum^*_{x^{\frac{k-1}k}<\bfN\eta\le x}h(\eta)=0.
\end{equation}

Denote by $\pi'(y,\msb,\fa)$ the number of primes $\theta\in\BZ[I]$ such that $\bfN\theta\le y$ and $\theta\mod\fa\in\msb\subseteq\bigl(\BZ[I]/\fa\bigr)^\times$.
Since $\theta\in\CP$ has positive imaginary part, we have
\[h(\eta)=\half\Bigl(\pi'(x/\bfN\eta,\msa_\eta,\fc_\eta)-
\pi'(\bfN\lambda_{k-1},\msa_\eta,\fc_\eta)\Bigr)+O(\sqrt x).\]
Here the error term origins from $-p$ with $p\equiv3\mod4$ rational prime, and the implicit constant is absolute.
By \eqref{eq:number-C_k'}, \eqref{eq:part-h-eta}, \eqref{eq:part2-h-eta} and the facts that
\[\sum^*_{M_1<\bfN\eta\le M_2}
\pi'(\bfN\lambda_{k-1},\msa_\eta,\fc_\eta)
\ll M_2\Li(M_2)=o\biggl(\frac{x(\log\log x)^{k-1}}{\log x}\biggr)\] and $M_2$ is of much small order than $x^{\frac14}$,  
we obtain
\begin{equation}\label{eq:num-C_k'(x,a,B)-M1-M2}
\#C_k'(x,\alpha,B)\sim\half\sum^*_{M_1<\bfN\eta\le M_2}
\pi'(x/\bfN\eta,\msa_\eta,\fc_\eta)
\end{equation}
with error term $o\bigl(\#C_k(x)\bigr)$.

By \cite[Theorem~6.1]{Lang1994}, we have an exact sequence
\begin{equation}\label{eq:exact-seq}
1\lra \BZ[I]^\times \lra (\BZ[I]/\fc_\eta)^{\times} \sto{\Phi}
I(\fc_\eta)/P_{\fc_\eta}\lra 1
\end{equation}
where $\Phi(\gamma)=(\gamma)\mod P_{\fc_\eta}$. 
Denote by $\pi(y,\msb,\fc)$ the number of prime ideals $\fp$ such that
$\bfN\fp\le y$ and $\fp\mod P_\fc\in\msb\subseteq I(\fc)/P_\fc$. Denote by $\mst_\eta=\Phi(\msa_\eta)$.  
Then 
\begin{equation}\label{eq:pi'-equals-pi}
\pi'(y,\msa_\eta,\fc_\eta)=\pi(y,\mst_\eta,\fc_\eta)
\qquad\text{and}\qquad\#\msa_\eta=\#\mst_\eta
\end{equation}
by noting that every prime ideal in a class of $\mst$ corresponds to exactly one primary prime element. 

Define
\[ \psi(y,\msb,\fc)=\sum_{\bfN\fa\le y\atop \fa\mod P_\fc\in \msb}\Lambda(\fc).\]
Then we have the standard asymptotic relation $\psi(y,\msb,\fc)\sim \log y\cdot\pi(y,\msb,\fc)$.

Therefore,
\begin{equation}\label{eq:num-2logx-C_k'(x,a,B)}
2\log x\cdot\#C_k'(x,\alpha,B)
\sim\sum^*_{M_1<\bfN\eta\le M_2} \psi(x/\bfN\eta,\mst_\eta,\fc_\eta)
\end{equation}
by \eqref{eq:num-C_k'(x,a,B)-M1-M2} and \eqref{eq:pi'-equals-pi}.
By the orthogonality of characters and the exact sequence \eqref{eq:exact-seq}, we get
\[\psi(y,\mst_\eta,\fc_\eta)=\frac4{\varphi(\eta)}\sum_{\chi} \psi(y,\chi)\sum_{\fa\mod P_{\fc_\eta}\in\mst_\eta} \ov{\chi(\fa)},\]
where $\chi$ runs over all characters of $I(\fc_\eta)/P_{\fc_\eta}$ and
\[\psi(y,\chi)=\sum_{\bfN\fa\le y} \Lambda(\fa)\chi(\fa).\]
Therefore,
\begin{equation}\label{eq:S1-S2}
2\log x\cdot\#C_k'(x,\alpha,B)\sim S_1+S_2,
\end{equation}
where
\[\begin{split}
S_1&=\sum^*_{M_1<\bfN\eta\le M_2}\frac{4\#\mst_\eta}{\varphi(\eta)} \psi(x/\bfN\eta,\chi_0),  \\
S_2&=\sum^*_{M_1<\bfN\eta\le M_2}\frac4{\varphi(\eta)}\sum_{\chi\neq\chi_0} \psi(x/\bfN\eta,\chi)\sum_{\fa\mod P_{\fc_\eta}\in\mst_\eta} \ov{\chi(\fa)}.
\end{split}\]

The main term is
\[\begin{split}
S_1&=2^{-k-\ell-2}\sum^*_{M_1<\bfN\eta\le M_2}\psi(x/\bfN\eta,\chi_0)
\qquad\text{by Lemma~\ref{lem:number-A-eta} and \eqref{eq:pi'-equals-pi}}\\
&\sim2^{-k-\ell-2}\sum^*_{M_1<\bfN\eta\le M_2}\log(x/\bfN\eta)\Li(x/\bfN\eta)\\
&\sim2^{-k-\ell-2}\log x\sum^*_{M_1<\bfN\eta\le M_2}
\Li(x/\bfN\eta)  \\
&\sim\frac{\log x\cdot\log\log x}{(k-1)\cdot2^{k+\ell+2}}\cdot\#T'_k(x)
\qquad\text{by Lemma~\ref{lem:Cremona}}\\
&\sim\frac{\log x\cdot\log\log x}{(k-1)\cdot2^{k\ell+3k+\binom k2}}\cdot\#C_{k-1}(x)
\qquad\text{by Lemma~\ref{lem:identify-Tk(x)} and \eqref{eq:number-Tk(x)}}\\
&\sim2^{-k\ell-3k-\binom k2}\log x\cdot\#C_k(x)
\qquad\text{by \eqref{eq:number-C-k}}.
\end{split}\]
By \eqref{eq:num-2logx-C_k'(x,a,B)} and Lemma~\ref{lem:identify-Ck}, this theorem is reduced to show that $S_2$ is an error term.

Denote by $\ff$ the conductor of the exceptional primitive conductor with $Z=256M_2$ in Page Theorem~\ref{pro:Page-theorem}.
Then $S_2=S_3+S_4$, where
\[\begin{split}
S_3&=\sum^*_{M_1<\bfN\eta\le M_2\atop \ff\mid\fc_\eta} \frac4{\varphi(\eta)}\sum_{\chi\neq\chi_0} \psi(x/\bfN\eta,\chi)\sum_{\fa\mod P_{\fc_\eta}\in\mst_\eta} \ov{\chi(\fa)},  \\
S_4&=\sum^*_{M_1<\bfN\eta\le M_2\atop \ff\nmid\fc_\eta} \frac4{\varphi(\eta)}\sum_{\chi\neq\chi_0} \psi(x/\bfN\eta,\chi)\sum_{\fa\mod P_{\fc_\eta}\in\mst_\eta} \ov{\chi(\fa)}.
\end{split}\]
We have
\[\begin{split}
S_3&\ll\sum^*_{M_1<\bfN\eta\le M_2\atop \ff\mid\fc_\eta}\psi(x/\bfN\eta,\chi_0)\ll x\sum^*_{M_1<\bfN\eta\le M_2\atop \ff\mid\fc_\eta}(\bfN\eta)^{-1}  \\
&=\frac x{\bfN\ff}\sum_{M_1<t\bfN\ff\le M_2}t^{-1}\sum^*_{\ff\mid \fc_\eta\atop\bfN\eta=t\bfN\ff}1\ll  \frac{x\log M_2}{\bfN\ff}.
\end{split}\]
By Page Theorem~\ref{pro:Page-theorem} for $Z=256M_2$, there is a positive constant $c_6$ such that the Siegel zero $\beta$ of the primitive character  with modulus $\ff$ has the property
\[\beta>1-\frac{c_6}{\log 256M_2}.\]
By Siegel Theorem~\ref{pro:Siegel-theorem} for $F=\BQ(I)$, there is a constant $c_4=c_4(2,1/200)>0$ such that
\[\beta\le 1-c_4(4\bfN\ff)^{-1/200}.\]
Therefore, $\bfN\ff\gg(\log M_2)^{100}$ and $S_3\ll x(\log M_2)^{-99}$ is an error term.

Since there is no Siegel zero in $S_4$, we can apply the explicit formula \eqref{eq:explicit-formula-of-psi-x-chi} with $T=(\bfN\eta)^4$ to all the $\psi(x/\bfN\eta,\chi)$ in $S_4$. Then we obtain
\[\begin{split}
\psi(x/\bfN\eta,\chi)\ll &x(\bfN\eta)^{-1}(\log x)^2 \exp
\biggl(-\frac{c_7\log(x/\bfN\eta)}{\log\bfN\eta}\biggr)\\
&+x(\bfN\eta)^{-5}(\log x)^2
+x^{1/4}(\bfN\eta)^{-1/4}\log(x/\bfN\eta)
\end{split}\]
and $S_4\ll S_5+S_6+S_7$, where
\[\begin{split}
S_5&=\sum^*_{M_1<\bfN\eta\le M_2\atop \ff\nmid\fc_\eta}
x(\bfN\eta)^{-1}(\log x)^2 \exp
\biggl(-\frac{c_7\log(x/\bfN\eta)}{\log\bfN\eta}\biggr), \\
&\ll x(\log x)^2\exp\bigl(-c_8(\log\log x)^{100}\bigr)\cdot
\sum^*_{M_1<\bfN\eta\le M_2\atop \ff\nmid\fc_\eta} (\bfN\eta)^{-1}  \\
&\ll x(\log x)^3\exp\bigl(-c_8(\log\log x)^{100}\bigr),\\
S_6&=\sum^*_{M_1<\bfN\eta\le M_2\atop \ff\nmid\fc_\eta}
x(\bfN\eta)^{-5}(\log x)^2\ll x(\log x)^2 M_1^{-3}\ll x(\log x)^{-200},  \\
S_7&=\sum^*_{M_1<\bfN\eta\le M_2\atop \ff\nmid\fc_\eta}
x^{1/4}(\bfN\eta)^{-1/4}\log(x/\bfN\eta)\ll x^{1/4}\log x\cdot M_2^{3/4}\ll x^{1/2}.
\end{split}\]
Hence $S_4$ is also an error term.
This finishes the proof.
\end{proof}

\section{Distribution result}
\label{Distribution result}

Assume that $\Sel_2(E/\BQ)\cong(\BZ/2\BZ)^2$.
Let $n=p_1\cdots p_k$ be an element in $\msq_k(x)$ with $p_1<\cdots<p_k$.
Then $n\in\msp_k(x)$ if and only if $h_4(n)=1$ and $h_8(n)\equiv\frac{d-1}4\bmod2$, where $d$ is a certain divisor of $n$.
As shown in the proof of Theorem~\ref{thm:main}(B), the rank of $\bfA=\bfA_n$ is $k-1$ or $k-2$.

Assume that $\rank\bfA=k-2$.
As shown in the proof of Theorem~\ref{thm:main}(B), $h_4(n)=1$ if and only if $\bfb_2\not\in\Im\bfA$.
In this case, $d=p_1^{s_1}\cdots p_k^{s_k}\equiv5\bmod8$, where $\Ker\bfA=\set{{\bf0},{\bf1},\bfd,\bfd+{\bf1}}$ and $\bfd=(s_1,\dots,s_k)^\rmT$.
We may assume that $s_k=0$.
By \cite[Theorem~3.3(ii)]{JungYue2011}, $h_8(n)=1$ if and only if
\begin{equation}\label{eq:d-d'-quartic-symbol}
\leg d{d'}_4\leg{d'}d_4=-1,
\end{equation}
where $d'=n/d$.

Assume that $\rank\bfA=k-1$.
Then $h_4(n)=1$, $\bfb_2\in\Im\bfA$ and $d=p_1^{s_1}\cdots p_k^{s_k}$, where $\bfA\bfd=\bfb_2$ and $\bfd=(s_1,\cdots,s_k)^\rmT$.
By \cite[Theorem~3.3(iii), (iv)]{JungYue2011}, $h_8(n)=1$ if and only if
\[\leg{2d}{d'}_4\leg{2d'}d_4=(-1)^{\frac{n-1}8}\]
where $d'=n/d$.

\begin{proof}[Proof of Theorem~\ref{thm:dist}]
For $k\ge 2$, let $\msb$ be the set of all symmetric $\bfB\in M_k(\BF_2)$ with rank $k-2$ and $\bfB{\bf1}={\bf0}$.
Let $\msi$ be the set of all vectors $\alpha=(\alpha_1,\dots,\alpha_k)$ with $\alpha_i\in\set{1,5,9,13}$ and $\alpha_1\cdots\alpha_k\equiv1\bmod8$.
Denote by $\msi_\bfB$ the set of all $\alpha\in\msi$ such that $\bfb(\alpha)\notin\Im\bfB$, where $\bfb(\alpha)=\biggl(\aleg{2}{\alpha_1},\dots,\aleg2{\alpha_k}\biggr)^\rmT$.
Since $\alpha_1\cdots\alpha_k\equiv1\bmod8$, we have $\bfb(\alpha)^\rmT{\bf1}=0$.
For any $\bfB\in\msb$ and $\alpha\in\msi_\bfB$, $C_k(x,\alpha,\bfB)$ is the set of all $n=p_1\cdots p_k\in \msp_k(x)$ satisfying
\begin{itemize}
\item $p_1<\cdots<p_k$ and $\bfA_n=\bfB$;
\item $p_i\equiv\alpha_i\bmod{16}$ for all $1\le i\le k$;
\item $\leg{p_i}{q_j}=1$ for all $1\le i\le k$ and $1\le j\le \ell$
\end{itemize}
by \eqref{eq:d-d'-quartic-symbol}.
Moreover, if $\bfB\in\msb$ and $\alpha\notin\msi_\bfB$, then $C_k(x,\alpha,\bfB)\cap\msp_k(x)=\emptyset$.
Therefore, the number $N_1(x)$ of those $n\in\msp_k(x)$ with $\rank\bfA_n=k-2$ is
\begin{equation}\label{num-N1(x)}
N_1(x)=\sum_{\bfB\in\msb}\sum_{\alpha\in\msi_\bfB}\#C_k(x,\alpha,\bfB)\sim 2^{-k\ell-3k-1-\binom k2}\cdot\#C_k(x)\cdot\sum_{\bfB\in\msb}\#\msi_\bfB
\end{equation}
by Theorem~\ref{thm:independence}.

Now we count the number of $\msi_\bfB$ with given $\bfB$.
Given $\bfb=(b_1,\cdots,b_k)^\rmT\not\in\Im\bfB$ with $\bfb^\rmT{\bf1}=0$, the number of $\alpha$ with $\bfb(\alpha)=\bfb$ is $2^k$.
This is because $\alpha_i=1,9$ if $b_i=0$ and $\alpha_i=5,13$ if $b_i=1$.
Since $\bfB$ is symmetric and $\bfB{\bf1}={\bf0}$, the size of $\Im\bfB\subset\CH_n:=\set{\bfu:{\bf1}^\rmT\bfu=0}$ is $2^{k-2}$.
If $\bfb^\rmT{\bf1}=0$ and $\rank(\bfB,\bfb)=k-1$, then $\bfb\in\CH_n-\Im\bfB$ has $2^{k-2}$ choices.
Consequently, $\#\msi_\bfB=2^{2k-2}$ and then
\[N_1(x)\sim 2^{-k\ell-k-3-\binom k2}\cdot\#C_k(x)\cdot\#\msb.\]

\begin{proposition}[{\cite{BrownCalkinJames2006}}]
Denote by $\msb_{k,r}$ the set of $k\times k$ symmetric matrices over $\BF_2$ with rank $r$. Then
\[\#\msb_{k,r}=u_{r+1} 2^{\binom{r+1}2}\cdot
\prod_{i=0}^{k-r-1}\frac{2^k-2^i}{2^{k-r}-2^i},\]
where $u_i$ is defined in Theorem~\ref{thm:dist}.
\end{proposition}
The left-top minor of $\bfB$ of order $k-1$ induces a bijection $\msb\to\msb_{k-1,k-2}$. So $\#\msb=\#\msb_{k-1,k-2}$ and we get
\[N_1(x)\sim 2^{-k\ell-k-3}(1-2^{1-k})u_{k-1}\cdot\#C_k(x).\]

The number $N_2(x)$ of $n\in\msp_k(x)$ with $\rank\bfA_n=k-1$ can be obtained similarly:
\[N_2(x)\sim 2^{-k-k\ell-2}u_k\cdot\#C_k(x).\]
We refer to our previous paper \cite{Wang2017} for more details.
This finishes the proof of this theorem.
\end{proof}

\noindent\textbf{Acknowledgement.}
The first author is supported by the National Natural Science Foundation of China (Grant No. 11801344) and Natural Science Foundation of Shaanxi Province (Grant No. 2020JQ-401).
The second author is supported by the National Natural Science Foundation of China (Grant No. 12001510).
The authors are greatly indebted to Professor Ye Tian for many instructions and suggestions.


\begin{thebibliography}{BCJ{\etalchar{+}}06}

\bibitem[BCJ{\etalchar{+}}06]{BrownCalkinJames2006}
Morgan~V. Brown, Neil~J. Calkin, Kevin James, Adam~J. King, Shannon Lockard,
  and Robert~C. Rhoades.
\newblock Trivial {S}elmer groups and even partitions of a graph.
\newblock {\em Integers}, 6:A33, 17, 2006.

\bibitem[Cas98]{Cassels1998}
J.~W.~S. Cassels.
\newblock Second descents for elliptic curves.
\newblock {\em J. Reine Angew. Math.}, 494:101--127, 1998.
\newblock Dedicated to Martin Kneser on the occasion of his 70th birthday.

\bibitem[CO89]{CremonaOdoni1989}
J.~E. Cremona and R.~W.~K. Odoni.
\newblock Some density results for negative {P}ell equations; an application of
  graph theory.
\newblock {\em J. London Math. Soc. (2)}, 39(1):16--28, 1989.

\bibitem[Dav80]{Davenport1980}
Harold Davenport.
\newblock {\em Multiplicative number theory}, volume~74 of {\em Graduate Texts
  in Mathematics}.
\newblock Springer-Verlag, New York-Berlin, second edition, 1980.
\newblock Revised by Hugh L. Montgomery.

\bibitem[Fen96]{Feng1996}
Keqin Feng.
\newblock Non-congruent numbers, odd graphs and the {B}irch-{S}winnerton-{D}yer
  conjecture.
\newblock {\em Acta Arith.}, 75(1):71--83, 1996.

\bibitem[Fog61]{Fogels1961}
E.~Fogels.
\newblock On the zeros of {H}ecke's {$L$}-functions. {I}.
\newblock {\em Acta Arith.}, 7(2):87--106, 1961.

\bibitem[Fog63]{Fogels1963}
E.~Fogels.
\newblock \"{U}ber die {A}usnahmenullstelle der {H}eckeschen
  {$L$}-{F}unktionen.
\newblock {\em Acta Arith.}, 8:307--309, 1963.

\bibitem[Fuj98]{Fujiwara1998}
Masahiko Fujiwara.
\newblock {$\theta$}-congruent numbers.
\newblock In {\em Number theory ({E}ger, 1996)}, pages 235--241. de Gruyter,
  Berlin, 1998.

\bibitem[HB94]{HeathBrown1994}
D.~R. Heath-Brown.
\newblock The size of {S}elmer groups for the congruent number problem. {II}.
\newblock {\em Invent. Math.}, 118(2):331--370, 1994.
\newblock With an appendix by P. Monsky.

\bibitem[Hec81]{Hecke1981}
Erich Hecke.
\newblock {\em Lectures on the theory of algebraic numbers}, volume~77 of {\em
  Graduate Texts in Mathematics}.
\newblock Springer-Verlag, New York-Berlin, 1981.
\newblock Translated from the German by George U. Brauer, Jay R. Goldman and R.
  Kotzen.

\bibitem[HR95]{HoffsteinRamakrishnan1995}
Jeffrey Hoffstein and Dinakar Ramakrishnan.
\newblock Siegel zeros and cusp forms.
\newblock {\em Internat. Math. Res. Notices}, 6:279--308, 1995.

\bibitem[IK04]{IwaniecKowalski2004}
Henryk Iwaniec and Emmanuel Kowalski.
\newblock {\em Analytic number theory}, volume~53 of {\em American Mathematical
  Society Colloquium Publications}.
\newblock American Mathematical Society, Providence, RI, 2004.

\bibitem[IR90]{IrelandRosen1990}
Kenneth Ireland and Michael Rosen.
\newblock {\em A classical introduction to modern number theory}, volume~84 of
  {\em Graduate Texts in Mathematics}.
\newblock Springer-Verlag, New York, second edition, 1990.

\bibitem[JY11]{JungYue2011}
Hwanyup Jung and Qin Yue.
\newblock 8-ranks of class groups of imaginary quadratic number fields and
  their densities.
\newblock {\em J. Korean Math. Soc.}, 48(6):1249--1268, 2011.

\bibitem[Lan18]{Landau1918}
Edmund Landau.
\newblock \"{U}ber {I}deale und {P}rimideale in {I}dealklassen.
\newblock {\em Math. Z.}, 2(1-2):52--154, 1918.

\bibitem[Lan94]{Lang1994}
Serge Lang.
\newblock {\em Algebraic number theory}, volume 110 of {\em Graduate Texts in
  Mathematics}.
\newblock Springer-Verlag, New York, second edition, 1994.

\bibitem[LT00]{LiTian2000}
Delang Li and Ye~Tian.
\newblock On the {B}irch-{S}winnerton-{D}yer conjecture of elliptic curves
  {$E_D\colon y^2=x^3-D^2x$}.
\newblock {\em Acta Math. Sin. (Engl. Ser.)}, 16(2):229--236, 2000.

\bibitem[Maz77]{Mazur1977}
B.~Mazur.
\newblock Modular curves and the {E}isenstein ideal.
\newblock {\em Inst. Hautes \'{E}tudes Sci. Publ. Math.}, 47:33--186 (1978),
  1977.
\newblock With an appendix by Mazur and M. Rapoport.

\bibitem[Maz78]{Mazur1978}
B.~Mazur.
\newblock Rational isogenies of prime degree (with an appendix by {D}.
  {G}oldfeld).
\newblock {\em Invent. Math.}, 44(2):129--162, 1978.

\bibitem[Ono96]{Ono1996}
Ken Ono.
\newblock Euler's concordant forms.
\newblock {\em Acta Arith.}, 78(2):101--123, 1996.

\bibitem[OZ14]{OuyangZhang2014}
Yi~Ouyang and ShenXing Zhang.
\newblock On non-congruent numbers with 1 modulo 4 prime factors.
\newblock {\em Sci. China Math.}, 57(3):649--658, 2014.

\bibitem[OZ15]{OuyangZhang2015}
Yi~Ouyang and Shenxing Zhang.
\newblock On second 2-descent and non-congruent numbers.
\newblock {\em Acta Arith.}, 170(4):343--360, 2015.

\bibitem[Rei34]{Redei1934}
L.~R\'{e}dei.
\newblock Arithmetischer {B}eweis des {S}atzes \"{u}ber die {A}nzahl der durch
  vier teilbaren {I}nvarianten der absoluten {K}lassengruppe im quadratischen
  {Z}ahlk\"{o}rper.
\newblock {\em J. Reine Angew. Math.}, 171:55--60, 1934.

\bibitem[Rho09]{Rhoades2009}
Robert~C. Rhoades.
\newblock 2-{S}elmer groups and the {B}irch-{S}winnerton-{D}yer conjecture for
  the congruent number curves.
\newblock {\em J. Number Theory}, 129(6):1379--1391, 2009.

\bibitem[Sil09]{Silverman2009}
Joseph~H. Silverman.
\newblock {\em The arithmetic of elliptic curves}, volume 106 of {\em Graduate
  Texts in Mathematics}.
\newblock Springer, Dordrecht, second edition, 2009.

\bibitem[Wan16]{Wang2016}
Zhang~Jie Wang.
\newblock Congruent elliptic curves with non-trivial {S}hafarevich-{T}ate
  groups.
\newblock {\em Sci. China Math.}, 59(11):2145--2166, 2016.

\bibitem[Wan17]{Wang2017}
Zhang~Jie Wang.
\newblock Congruent elliptic curves with non-trivial {S}hafarevich-{T}ate
  groups: distribution part.
\newblock {\em Sci. China Math.}, 60(4):593--612, 2017.

\end{thebibliography}

\newcommand{\etalchar}[1]{$^{#1}$}

\end{document}